\documentclass[12pt]{amsart}
\usepackage{latexsym,amssymb,amsfonts,amsmath, amsthm, amscd, euscript, amsfonts, MnSymbol}

\usepackage[latin1]{inputenc}
\usepackage[english]{babel}

\usepackage{enumerate}
\usepackage{verbatim}
\usepackage{multicol}
\usepackage{graphicx}
\usepackage{epsfig}
\usepackage{color}
\usepackage{fullpage}
\usepackage{cases}

\newtheorem{theorem}{Theorem}
\newtheorem{claim}{Claim}
\newtheorem{subclaim}{Subclaim}

\newtheorem{remark}{Remark}
\newtheorem{lemma}[theorem]{Lemma}

\newtheorem{proposition}[theorem]{Proposition}
\newtheorem{corollary}[theorem]{Corollary}
\newtheorem{example}{Example}

\newcommand{\age}{Age}
\newcommand{\inc}{Inc}
\newcommand{\comp}{{Comp}}

\newcommand{\iinc}{inc}

\newcommand{\NN}{{\mathbb N}}

\begin{document}

\title{Metric properties of incomparability graphs with an emphasis on paths}
\author [M.Pouzet]{Maurice Pouzet}
\address{Univ. Lyon, Universit\'e Claude-Bernard Lyon1, CNRS UMR 5208, Institut Camille Jordan, 43, Bd. du 11 Novembre 1918, 69622
Villeurbanne, France et Department of Mathematics and Statistics, University of Calgary, Calgary, Alberta, Canada}
\email{pouzet@univ-lyon1.fr }

\author[I.Zaguia]{Imed Zaguia*}\thanks{*Corresponding author. Supported by Canadian Defence Academy Research Program, NSERC and LABEX MILYON (ANR-10-LABX-0070) of Universit\'e de Lyon within the program ''Investissements d'Avenir (ANR-11-IDEX-0007'' operated by the French National Research Agency (ANR)}
\address{Department of Mathematics \& Computer Science, Royal Military College of Canada,
P.O.Box 17000, Station Forces, Kingston, Ontario, Canada K7K 7B4}
\email{zaguia@rmc.ca}

\date{\today}


\begin{abstract}We describe some metric properties of incomparability graphs. We consider the problem of the existence of  infinite paths, either  induced or isometric, in the incomparability graph of a poset. Among other things, we show that if the  incomparability graph of a poset is connected and has infinite diameter, then it contains an infinite induced path. Furthermore, if the diameter of the set of vertices of degree at least $3$ is infinite, then  the graph  contains  as an induced subgraph  either a comb or a kite. 
\end{abstract}

\maketitle

\input{epsf}

\section{Introduction and presentation of the results}
In this paper, we highlight the special properties of incomparability graphs by considering the behavior of paths. We consider the problem of the existence of  infinite  paths, either  induced or isometric, in the incomparability graph of a poset. We apply one of our results in the theory of hereditary classes of certain permutation classes that are well quasi ordered by embeddability.
\vskip0.5cm

The graphs we consider are undirected, simple and have no loops. That is, a {\it graph} is a pair $G:=(V, E)$, where $E$ is a subset of $[V]^2$, the set of $2$-element subsets of $V$. Elements of $V$ are the {\it vertices} of $G$ and elements of $ E$ its {\it edges}. The graph $G$ be given, we denote by $V(G)$ its vertex set and by $E(G)$ its edge set. The {\it complement} of a graph $G=(V,E)$ is the graph $G^c$ whose vertex set is $V$ and edge set $E^c:=[V]^2\setminus  E$.

Throughout, $P :=(V, \leq)$ denotes an ordered set (poset), that is a set $V$ equipped with a binary relation $\leq$ on $V$ which is reflexive, antisymmetric and transitive. We say that two elements $x,y\in V$ are \emph{comparable} if $x\leq y$ or $y\leq x$, otherwise  we say they are \emph{incomparable}.  The \emph{comparability graph}, respectively the \emph{incomparability graph}, of a poset $P:=(V,\leq)$ is the undirected graph, denoted by $\comp(P)$, respectively $\inc(P)$, with vertex set $V$ and edges the pairs $\{u,v\}$ of comparable distinct vertices (that is, either $u< v$ or $v<u$) respectively incomparable vertices.

A result of Gallai,  1967 \cite {gallai},  quite famous and nontrivial, characterizes comparability graphs among graphs in terms of obstructions: a graph $G$ is the  comparability graph  of a poset if and only if  it does not contain as an induced subgraph a graph belonging to a minimal list of finite graphs.
Since the complement of a comparability graph is an incomparability graph, Gallai's result yields a similar characterization of incomparability graphs. In this paper,  we consider  incomparability graphs  as metric spaces by means of the distance of the shortest path. The metric properties of a  graph,  notably of an incomparability graph, and metric properties of  its complement seem  to be far apart. In general,  metric properties of graphs are based on paths and cycles. It should be noted that  incomparability graphs have no induced cycles of length at least five (\cite {gallai}; for a short proof see after Lemma \ref{lem:inducedpath}) while comparability graphs have no induced odd cycles but can have arbitrarily large induced even cycles.

In the sequel,  we will illustrate the specificity of the metric properties of incomparability graphs by  emphasising  the properties of paths. 

We start with few definitions. Let $G:=(V, E)$ be a graph.  If $A$ is a subset of $V$, the graph $G_{\restriction A}:=(A,  E\cap [A]^2)$ is the \emph{graph induced by $G$ on $A$}.  A \emph{path} is a graph $\mathrm P$ such that there exists a one-to-one map $f$ from the set $V(\mathrm P)$ of its vertices into an interval $I$ of the chain $\NN$ of nonnegative integers in such a way that $\{u,v\}$ belongs to $E(\mathrm P)$, the set of edges of $\mathrm P$,  if and only if $|f(u)-f(v)|=1$ for every $u,v\in V(\mathrm P)$.  If $I$ is finite, say $I=\{1,\dots,n\}$, then we denote that path by $\mathrm P_n$; its \emph{length} is $n-1$ (so, if $n=2$, $\mathrm P_2$ is made of a single edge, whereas if $n=1$, $\mathrm P_1$ is a single vertex). We denote by $\mathrm P_\infty$ the one way infinite path i.e. $I=\NN$.
If $x,y$ are two vertices of a graph $G:= (V, E)$, we denote by $d_G(x,y)$ the length of a shortest path joining $x$ and $y$ if any, and $d_G(x,y):= \infty$ otherwise. This defines a distance on $V$, the \emph{graphic distance}. A graph is \emph{connected} if any two vertices belong to some  path. The \emph{diameter} of $G$, denoted by  $\delta_{G}$, is the supremum of the set  $\{d_G(x,y) : x,y\in V\}$. If $A$ is a subset of $V$, the graph $G'$ induced by $G$ on $A$  is an \emph{isometric subgraph} of $G$  if $d_{G'}(x,y)=d_G(x,y)$ for all $x,y\in A$. The supremum of the length of induced finite paths of $G$, denoted by $D_G$,  is sometimes called the (induced) \emph{detour} of $G$ \cite{buckley-harary}.\\

The main results of the paper are presented in the next four subsections.  Section \ref{sec:application} is devoted to an application of one of our main results (Theorem \ref{thm:infinitepath-kite}). The remaining sections contain intermediate results and proofs of our main results.

\subsection{Induced paths of arbitrarily large length in incomparability graphs and in arbitrary graphs}
We now consider the question of the existence  of infinite  induced paths in incomparability graphs with infinite detour. In order to state our main result of this subsection we need to introduce the notions of direct sum and complete sum of graphs.
Let $G_n:=(V_n,E_n)$ for $n\in \NN$ be a family of graphs having pairwise disjoint vertex sets. The \emph{direct sum} of $(G_n)_{n\in \NN}$, denoted $\oplus_n G_n$, is the graph whose vertex set is $\bigcup_{n\in \NN}V_n$ and edge set $\bigcup_{n\in \NN}E_n$. The \emph{complete sum} of $(G_n)_{n\in \NN}$, denoted $\sum_n G_n$, is the graph whose vertex set is $\bigcup_{n\in \NN}V_n$ and edge set $\bigcup_{i\neq j}\{\{v,v'\} : v\in V_i \wedge v'\in V_j\}\cup  \bigcup_{n\in \NN}E_n$.\\

A necessary condition for the existence of an infinite induced path in a graph is to have infinite detour. On the other hand, the graphs consisting of the direct sum of finite paths of arbitrarily large length and the complete sum of finite paths of arbitrarily large length are (incomparability) graphs with  infinite detour and yet do not have an infinite induced path. We should mention that \emph{in the case of incomparability graphs, having infinite detour is equivalent to having a direct sum or a complete sum of finite paths of arbitrarily large length}. This is Theorem 2 from \cite{pouzet-zaguia20}.


\begin{theorem}[\cite{pouzet-zaguia20}]\label{thm:pouzet-zaguia-pathmin}Let $G$ be the incomparability graph of a poset. Then $G$ contains induced paths of arbitrarily large length if and only if $G$ contains $\sum_{n\geq 1} \mathrm P_n$ or $\oplus_{n\geq 1} \mathrm P_n$ as an induced subgraph.
\end{theorem}


For general graphs, the statement of Theorem \ref{thm:pouzet-zaguia-pathmin} is false. Indeed, in \cite{pouzet-zaguia20} we exhibited uncountably many graphs of cardinality $\aleph_0$, containing finite induced paths of unbounded  length  and neither  a direct sum nor a complete sum of finite paths of unbounded length. In particular, these graphs do not have an infinite induced path. \\ 

%
%
%

In the case of incomparability graphs of posets coverable by two chains, having infinite detour is equivalent to the existence of an infinite induced path. Our first result is this.

\begin{theorem}\label{thm:widthtwo}Let $P$ be a poset coverable by two chains (that is totally  ordered sets). If $\inc(P)$, the incomparability graph of $P$, is connected then the following properties are equivalent:
\begin{enumerate}[(i)]
\item $\inc(P)$ contains the direct sum of induced paths of arbitrarily large length;
\item the detour of $\inc(P)$ is infinite;
\item the diameter of $\inc(P)$ is infinite;
\item $\inc(P)$ contains an infinite induced path.
\end{enumerate}
\end{theorem}


A proof of Theorem \ref{thm:widthtwo} will be provided in Section \ref{proof:thm:widthtwo}.

The implication  $(i)\Rightarrow (iv)$ of Theorem \ref{thm:widthtwo} becomes false if the condition "coverable by two chains" is dropped (see Figure \ref{width-three} for an example). Indeed,

\begin{example}\label{thm:infi-detour-no-path}There exists a poset with no infinite antichain whose incomparability graph is connected and embeds the direct sum of finite induced paths of arbitrarily large length and yet does not have an infinite induced path (See Figure \ref{width-three}).
\end{example}

Example \ref{thm:infi-detour-no-path} and a proof that it verifies the required properties will be given in Section \ref{proof:thm:infi-detour-no-path}.


\begin{figure}[h]
\begin{center}
\leavevmode \epsfxsize=2.4in \epsfbox{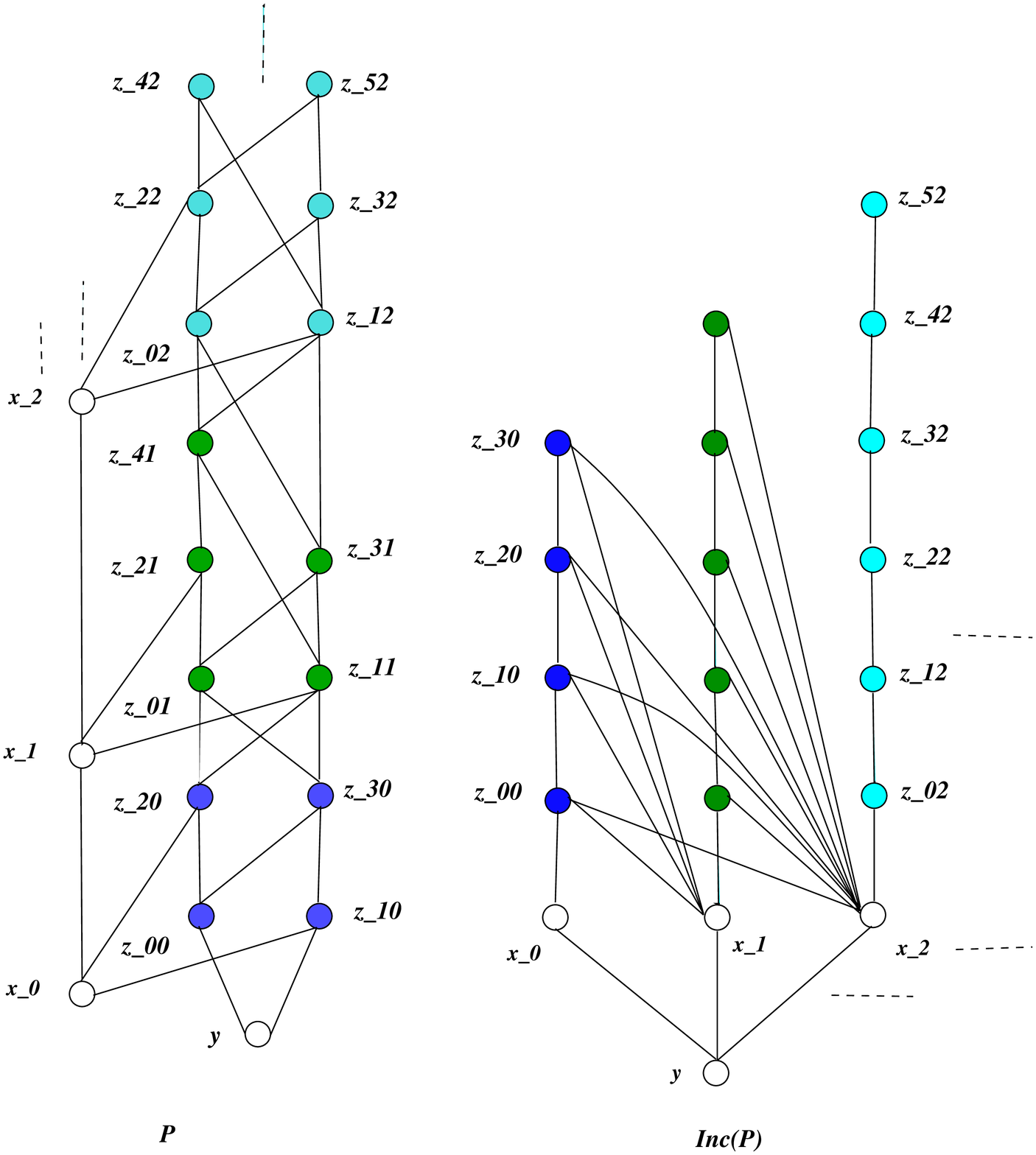}
\end{center}
\caption{The Hasse diagram of a poset of width three and its incomparability graph that has a vertex $y$ of infinite induced detour but no infinite induced path.}
\label{width-three}
\end{figure}

\subsection{Infinite induced  paths, combs and kites}
We now consider the question of the existence  of infinite  induced paths in incomparability graphs with infinite diameter. In order to state our main result of this subsection we need to introduce two types of graphs: comb and kite.


Let us recall that a  graph  $G:=(V,E)$ is a \emph{caterpillar} if the graph obtained by removing from $V$ the vertices of degree one is a path (finite or not, reduced to one vertex or empty). A \emph{comb} is a caterpillar such
that every vertex is adjacent to at most one vertex of degree one. Incidentally, a path on three vertices is not a comb. It should be mentioned that caterpillars are incomparability graphs of interval orders coverable by two chains (see Lemma 14 of \cite{zaguia2008}).

We now give the definition of a \emph{kite}. This is a graph obtained from an infinite path $\mathrm P_\infty :=(x_i)_{i\in \NN}$ by adding a new set of vertices $Y$ (finite or infinite). We distinguish three types of kites (see Figure \ref{fig:comb-kite}) depending on how the vertices of $Y$ are adjacent to the vertices of $\mathrm P_\infty$.

A \emph{kite of type $(1)$}: every vertex of $Y$ is adjacent to exactly two vertices of $\mathrm P_\infty$ and these two vertices are consecutive in $\mathrm P_\infty$. Furthermore, two distinct vertices of $Y$ share at most one common neighbour in $\mathrm P_\infty$.

A \emph{kite of type $(2)$}: every vertex of $Y$ is adjacent to exactly three vertices of $\mathrm P_\infty$ and these three vertices must be consecutive in $\mathrm P_\infty$. Furthermore, for all $x,x'\in Y$, if $x$ is adjacent to $x_i,x_{i+1},x_{i+2}$ and $x'$ is adjacent to $x_{i'},x_{i'+1},x_{i'+2}$  then $i+2\leq i'$ or $i'+2\leq i$.

A \emph{kite of type $(3)$}: every vertex of $Y$ is adjacent to exactly two vertices of $\mathrm P_\infty$ and these two vertices must be at distance two in $\mathrm P_\infty$. Furthermore, for all $x,x'\in X$, if $x$ is adjacent to $x_i$ and $x_{i+2}$ and $x'$ is adjacent to $x_{i'}$ and $x_{i'+2}$ then $i+2\leq i'$ or $i'+2\leq i$.


\begin{figure}[h]
\begin{center}
\leavevmode \epsfxsize=3.5in \epsfbox{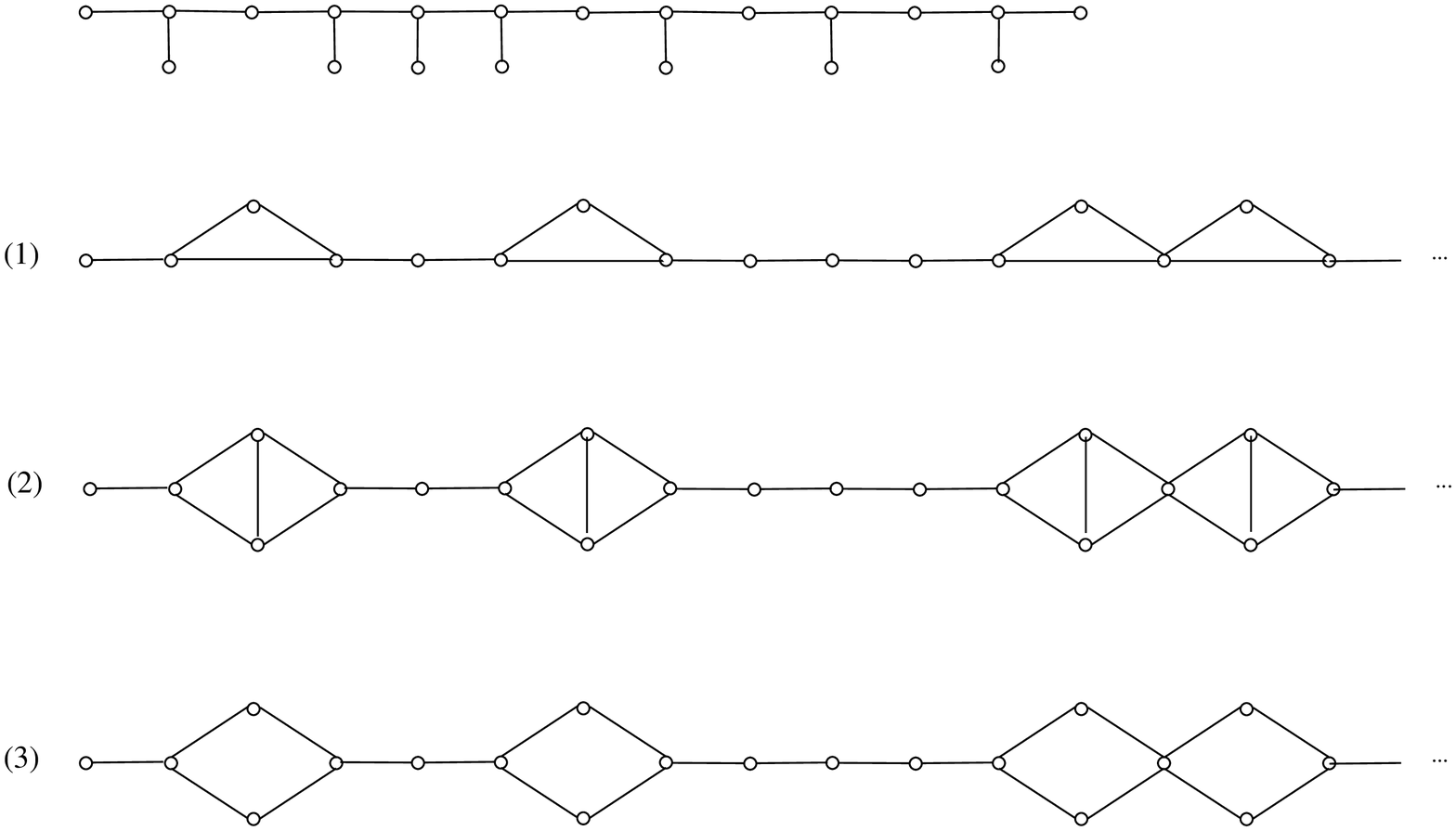}
\end{center}
\caption{A comb and the three types of kites.}
\label{fig:comb-kite}
\end{figure}

\begin{theorem}\label{thm:infinitepath-kite}If $G$ is a connected incomparability graph with infinite diameter. Then
\begin{enumerate}[$(1)$]
  \item Every vertex of $G$ has an induced path of infinite diameter starting at it.
  \item If the set of vertices of degree at least $3$ in $G$ has infinite diameter, then $G$ contains an induced comb or an induced kite having an infinite diameter and infinitely many vertices of degree at least $3$.
\end{enumerate}
\end{theorem}

Theorem \ref{thm:infinitepath-kite} will be proved in Section \ref{section:proof-thm:infinitepath-kite} (an important ingredient of its proof is Theorem \ref{thm:orderconvex} below).

\subsection{Infinite isometric paths in incomparability graphs}
A basic result about the existence of an infinite isometric path in a graph is K\"onig's lemma \cite{konig}. Recall that a graph is \emph{locally finite} if every vertex has a finite degree.

 \begin{theorem}[\cite{konig}] \label{thn:konig}Every connected, locally finite, infinite graph contains an isometric infinite path.
 \end{theorem}

Moreover,

\begin{theorem}\label{thm:polat} If  a connected graph $G$ has an infinite isometric path, then every vertex has an isometric path
starting at it.
\end{theorem}

Theorem \ref{thm:polat} was proved by Watkins in the case of locally finite graphs (see \cite{watkins}, Lemma 3.2). The general case is contained in Theorem 3.5 and Lemma 3.7 of \cite{polat}.\\

A necessary condition for a graph to have an infinite isometric path is to have infinite diameter. Note that a graph has an infinite diameter if and only if it has finite isometric paths of  arbitrarily large length. The existence of such paths does not necessarily imply the existence of an infinite isometric path even if the graph is connected.



\begin{figure}[h]
\begin{center}
\leavevmode \epsfxsize=2in \epsfbox{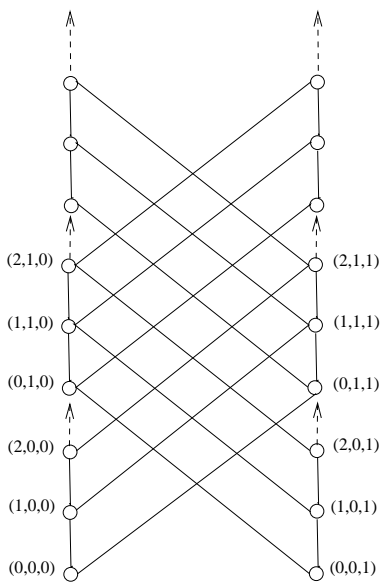}
\end{center}
\caption{The Hasse diagram of a poset of width two whose incomparability graph is connected, has infinite diameter but no infinite isometric path.}
\label{width-two}
\end{figure}

\begin{example}\label{thm:noisometric} There exists a poset coverable by two chains whose incomparability graph is connected, having infinite diameter and no isometric infinite path (see Figure \ref{width-two}).
\end{example}

We provide Example \ref{thm:noisometric} and a proof that it verifies the required properties in Section \ref{section:proof-thm:noisometric}.



%
%


We obtain a positive result in the case of incomparability graphs of interval orders with no infinite antichains. A poset $P$ is an {\it interval order} if $P$ is isomorphic to a subset $\mathcal J$ of the set $Int(C)$ of non-empty intervals of a chain $C$, ordered as follows: if $I, J\in
Int(C)$, then
\begin{equation}\label{ordre-sur-intervalles}
I<J \mbox{  if  } x<y  \mbox{  for every  } x\in I  \mbox{  and
every  } y\in J.
\end{equation}

Interval orders were considered in Fishburn \cite{fishburn-book,fishburn} and Wiener \cite{wiener} in relation to the theory of measurement.


\begin{theorem}\label{thm:intervalorder-isometric}If $P$ is an interval order with no infinite antichains so that $\inc(P)$ is connected and has infinite diameter, then $\inc(P)$ has an infinite isometric path.
\end{theorem}

The proof of Theorem \ref{thm:intervalorder-isometric} will be provided in Section \ref{section:intervalorders}.

The conclusion of Theorem \ref{thm:intervalorder-isometric} becomes false if the condition "no infinite antichains" is removed. Indeed,

\begin{example}\label{thm:intervalorder-non-isometric} There exists an interval order whose incomparability graph is connected,  has an infinite diameter and no  infinite isometric path.
\end{example}

Example \ref{thm:intervalorder-non-isometric} and a proof that it verifies the required properties will be provided in Section~\ref{section:intervalorders}.


\subsection{Convexity and isometry of metric balls in incomparability graphs}
In this subsection we compare the notions of order convexity and metric convexity with respect to the distance on the incomparability  graph of a poset. Before stating our result we need few definitions.

An \emph{initial segment} of a  poset $P:=(V,\leq)$ is any subset $I$ of $V$ such that $x\in V$, $y\in I$ and $x\leq y$ imply $x\in I$. If $X$ is a subset of $V$, the set $\downarrow X:=\{y\in P: y\leq x \; \text{for some}\; x\in X\}$ is the least initial segment containing $X$, we say that it is \emph{generated} by $X$. If $X$ is a one element set, say $X=\{x\}$, we denote by $\downarrow x$, instead of $\downarrow X$, this initial segment and say that it is \emph{principal}. \emph{Final segments} are defined similarly.

Let $P:=(V,\leq)$ be a poset. A subset $X$ of $V$ is \emph{order convex} or \emph{convex} if for all $x,y\in X$, $[x,y]:=\{z : x\leq z\leq y\}\subseteq X$. For instance, initial and final segments of $P$ are convex. Note that any intersection of convex sets is also convex. In particular, the intersection of all convex sets containing $X$, denoted $Conv_P(X)$, is convex. This is the smallest convex set containing $X$. Note that
\[Conv_P(X)=\{z\in P : x\leq z\leq y \mbox{ for some } x,y\in X\}=\downarrow X\cap \uparrow X.\]

Let $G:=(V,E)$ be a graph. We equip it with the graphic distance $d_G$. A \emph{ball} is any subset $B_G(x, r):= \{y\in V: d_G(x,y)\leq r\}$ where $x\in V, r\in \NN$. A subset of $V$ is \emph{convex} w.r.t. the distance $d_G$ if this is an intersection of balls. The \emph{least convex subset} of $G$ containing $X$ is
\[Conv_{G}(X):=\displaystyle \bigcap_{X\subseteq B_G(x,r)}B_G(x,r).\]

Let $X\subseteq V$ and $r\in \NN$. Define
\[B_G(X,r):=\{v\in V : d_G(v,x)\leq r \mbox{ for some } x\in X\}.\]

With all needed definitions in hand we are now ready to state the following theorem.

\begin{theorem} \label{thm:orderconvex}Let $P:=(V,\leq)$ be a poset, $G$ be its incomparability graph,  $X\subseteq V$ and $r\in \NN$.
\begin{enumerate}[$(a)$]
\item If $X$ is an  initial segment, respectively a final segment, respectively an order convex subset of $P$ then $B_G(X,r)$ is an initial segment, respectively a final segment, respectively an order convex subset of $P$. In particular, for all $x\in V$ and $r\in \NN$, $B_G(x,r)$ is order convex;
\item \label{lem:convex-connected}  If $X$ is order convex then the graph induced by $G$ on $B_G(X,r)$ is an isometric subgraph of $G$. In particular, if $X$ is  included into a connected component  of $G$ then the graph induced by $G$ on $B_G(X,r)$ is connected.
\end{enumerate}
\end{theorem}

It follows from Theorem \ref{thm:orderconvex} that \emph{every ball in an incomparability graph $G$ of a poset is order convex and that the graph induced on it is an  isometric subgraph of $G$}.

The proof of Theorem \ref{thm:orderconvex} is provided in Section \ref{section:proof-thm-orderconvex}.

\subsection{An application of Theorem \ref{thm:infinitepath-kite} in the theory of well quasi order}\label{sec:application}
The purpose of this subsection is to provide an application of Theorem \ref{thm:infinitepath-kite} in the theory of well quasi order. Let us first recall some notions from the Theory of Relations \cite{fraissetr}. A graph $G$ is \emph{embeddable} in a graph $G'$ if $G$ is isomorphic to an induced subgraph of $G'$. The embeddability relation is a quasi order on the class of graphs. A class $\mathcal C$ of  graphs, finite or not, is \emph{hereditary} if it contains every graph which embeds in some member of $\mathcal C$. The \emph{age} of a graph $G$ is the collection of finite graphs, considered up to isomorphy,  that embed in $G$ (or alternatively, that are isomorphic to some induced subgraph of $G$).   We recall that an age of finite graphs, and more generally a class of finite graphs, is \emph{well quasi ordered} (w.q.o. for short) if it contains no infinite antichain, that is an infinite set  of graphs $G_n$ pairwise incomparable with respect to embeddability. There are several results about   w.q.o. hereditary classes of graphs, see for examples \cite{korpelainen-lozin-razgon,korpelainen-lozin}, \cite{lozin-mayhill} and \cite{oudrar}.

We recall that a graph $G:= (V, E)$ is a \emph{permutation graph} if there is a linear order $\leq $ on $V$ and a permutation $\sigma$ of $V$ such that the edges of $G$ are the pairs  $\{x, y\}\in [V]^2$ which are reversed by $\sigma$. The study of permutations graphs became an important topic due to the Stanley-Wilf Conjecture, formulated independently by Richard P. Stanley and Herbert Wilf in the late 1980s, and solved positively by Marcus and   Tard\"os \cite{marcus-tardos} 2004.  It was proved by Lozin and Mayhill 2011\cite{lozin-mayhill} that a hereditary class of finite bipartite permutation graphs is w.q.o. by embeddability if and only there is a bound on the length of the double ended forks (see Figure \ref{fig:doublefork}) it may contain (for an alternative proof see \cite{pouzet-zaguia-wqo20}).   In \cite{pouzet-zaguia-wqo20}, we extend results of Lozin and Mayhill \cite{lozin-mayhill} and present an almost exhaustive list of properties of w.q.o. ages of bipartite permutation graphs. One of our results is a positive answer, in the case of an age of bipartite permutation graphs, to a long standing unsolved question by the first author, of whether the following equivalence is true in general: an age is not w.q.o. if and only if it contains $2^{\aleph_0}$ subages (see subsection I-4 Introduction \`a la  comparaison des \^ages, page 67, \cite{pouzet-israel}). This result, Theorem \ref{thm:3} below, is a consequence of (2) of Theorem \ref{thm:infinitepath-kite}.


\begin{figure}[h]
\begin{center}
\leavevmode \epsfxsize=3in \epsfbox{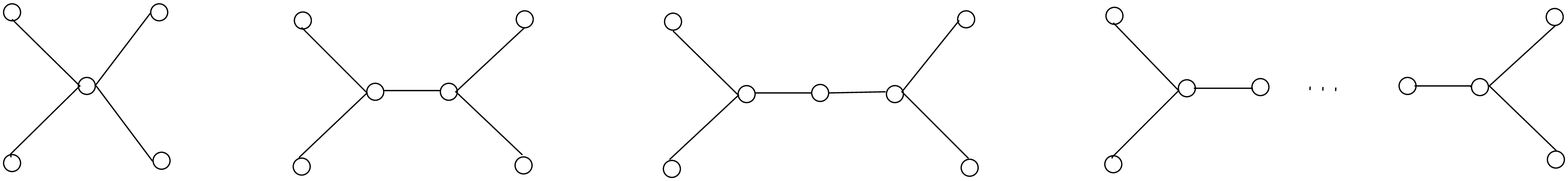}
\end{center}
\caption{Double-ended forks: an antichain of finite graphs with respect to embeddability.}
\label{fig:doublefork}
\end{figure}

\begin{theorem}[\cite{pouzet-zaguia-wqo20}] \label{thm:3}Let $\mathcal{C}$ be an age that consists of finite bipartite permutation graphs. Then $\mathcal{C}$ is not w.q.o. if and only if it contains the age of a direct sum $\bigoplus_{i\in I} \mathrm{DF}_i$ of double ended forks of arbitrarily large length for some infinite subset $I$ of $\NN$. In particular,  if $\mathcal{C}$ is not w.q.o.,  it contains $2^{\aleph_0}$ subages which are not w.q.o.
\end{theorem}


A proof is given in \cite{pouzet-zaguia-wqo20}. For completeness we provide the proof here.

\begin{proof}
The set of double-ended forks forms an infinite antichain, hence if $\mathcal C$ contains the direct sum $\bigoplus_{i\in I} \mathrm{DF}_i$ of double ended forks of arbitrarily large length for some infinite subset $I$ of $\NN$, it is not w.q.o. Conversely, suppose $\mathcal C$ is not w.q.o. Then it embeds double-ended forks of unbounded length. This important result is due Lozin and Mayhill (see Theorem 7 in \cite{lozin-mayhill}). Let $G$ be a graph with $\age(G)= \mathcal C$. We consider two  cases:\\
 $(1)$ Some connected component of $G$, say $G_i$, embeds double forks of unbounded length. In this case, the detour of $G_i$, that is the supremum of the lengths of induced paths in $G_i$,  is unbounded. Since $G_i$ is the incomparability graph of a poset of width at most two, its diameter is unbounded  (See Corollary \ref{cor:detour}). In fact, since the vertices of degree $3$ in the forks are end vertices of induced paths, the diameter of the set of vertices of degree $3$ in $G_i$ is unbounded.  Thus from $(2)$ of Theorem \ref{thm:infinitepath-kite},  $G_i$ embeds an induced caterpillar or an induced kite  with infinitely many vertices of degree at least $3$.  Since $G$ is bipartite, it can only embed a kite of type $(3)$. As it is easy to see, this caterpillar or that kite embeds a direct sum $\bigoplus_{i\in I} \mathrm {DF}_i$ of double-ended forks of arbitrarily large length, as required. \\
 $(2)$ If the first case does not hold, there are infinitely many connected components $G_i$, each embedding some double-ended fork $ \mathrm {DF}_i$, and the length of these double-ended forks is unbounded. This completes the proof of Theorem \ref{thm:3}.
\end{proof}


The paper is organised as follows. In Section \ref{section:prequisite} we present some prerequisites on graphs and posets. In Section \ref{section:fund-lemma} we state a fundamental lemma on paths in incomparability graphs and some consequences. In Section \ref{posetswidth2} we present few metric properties of posets of width $2$. In Section \ref{proof:thm:widthtwo} we present the proof of Theorem \ref{thm:widthtwo}. In Section  \ref{proof:thm:infi-detour-no-path} we present Example \ref{thm:infi-detour-no-path}. In Section \ref{section:convexity} we present various metric properties of incomparability graphs. In Section \ref{section:proof-thm-orderconvex} we present a proof of Theorem \ref{thm:orderconvex} and some consequences. In Section \ref{section:proof-thm:infinitepath-kite} we give a proof of Theorem \ref{thm:infinitepath-kite} (an important ingredient of the proof is Theorem \ref{thm:orderconvex}). In Section \ref{section:proof-thm:noisometric} we present Example \ref{thm:noisometric}. Finally a proof of Theorem \ref{thm:intervalorder-isometric} and Example \ref{thm:intervalorder-non-isometric} are provided in Section \ref{section:intervalorders}.

\section{Graphs and Posets}\label{section:prequisite}

\subsection{Posets}Throughout, $P :=(V, \leq)$ denotes an ordered set (poset).
The \emph{dual} of $P$ denoted $P^{*}$ is the order defined on $V$ as follows: if $x,y\in
V$, then $x\leq y$ in $P^{*}$ if and only if $y\leq x$ in $P$. Let $P :=(V, \leq)$ be a poset. We recall that two elements $x,y\in V$ are \emph{comparable} if $x\leq y$ or $y\leq x$, otherwise,  we say they are \emph{incomparable}, denoted $x\parallel y$.  A set of pairwise comparable elements is called a \emph{chain}. On the other hand, a set of pairwise incomparable elements is called an \emph{antichain}. The \emph{width} of a poset is the maximum cardinality  of its antichains (if the maximum does not exist, the width is set to be infinite). Dilworth's celebrated theorem on finite posets \cite{dilworth} states that the maximum cardinality of an antichain in a finite poset equals the minimum number of chains needed to cover the poset. This result remains true even if the poset is infinite but has finite width. If  the poset $P$ has width $2$ and the incomparability graph of $P$ is connected, the partition of $P$ into two chains is unique (picking  any vertex $x$,  observe that the set of  vertices at odd distance from $x$ and the set of vertices at even distance from $x$ form a partition into two chains). According to  Szpilrajn \cite{szp},  every order on  a set  has a linear extension. Let $P:=(V,\leq)$ be a poset.  A \emph{realizer} of $P$ is a family $\mathcal{L}$ of linear extensions of the order of $P$ whose intersection is the order of $P$. Observe that the set of all linear extensions of $P$ is a realizer of $P$. The \emph{dimension} of $P$, denoted $dim(P)$, is the least cardinal $d$ for which there exists a realizer of cardinality $d$ \cite{dushnik-miller}. It follows from the  Compactness Theorem of First Order Logic that an order is intersection of  at most $n$ linear orders ($n\in \NN$) if and only if every finite restriction of the order has this property. Hence the class of posets with dimension at most $n$ is determined by a set of finite obstructions, each   obstruction is a poset $Q$  of dimension $n+1$  such that  the deletion of any element of $Q$ leaves a poset of dimension $n$; such a poset is said \emph{critical}.  For $n\geq 2$ there are infinitely many critical posets of dimension $n+1$. For $n=2$ they have been  described by Kelly\cite{kelly77}; beyond, the task is considered as hopeless.

\subsubsection{Comparability and incomparability graphs, permutation graph}  A graph $G:= (V, E)$ is a \emph{comparability graph} if the edge set is the set of comparabilities of some order on $V$. From the Compactness Theorem of First Order Logic, it follows that a graph is a comparability graph if and only if every finite induced subgraph is a comparability graph. Hence, the class of comparability graphs is determined by a set of finite obstructions. The complete list of minimal obstructions was determined by Gallai \cite{gallai}. A graph $G:= (V, E)$ is a \emph{permutation graph} if there is a linear order $\leq $ on $V$ and a permutation $\sigma$ of $V$ such that the edges of $G$ are the pairs  $\{x, y\}\in [V]^2$ which are reversed by $\sigma$. Denoting by $\leq_{\sigma}$ the set of oriented pairs $(x, y)$ such that $\sigma(x) \leq \sigma (y)$, the graph is the comparability graph of the poset whose order is the intersection of $\leq$ and the opposite of $\leq_{\sigma}$.  Hence, a permutation graph is the comparability graph of  an order intersection of two linear orders, that is the comparability graph of an order of dimension at most two \cite{dushnik-miller}. If the graph is finite, the converse holds. Hence, as it is well known, a finite graph $G$ is a permutation graph if and only if $G$ and $G^c$ are comparability graphs \cite{dushnik-miller}; in particular, a finite graph  is a permutation graph if and only if its complement is a permutation graph.  Via the Compactness Theorem of First Order Logic, an infinite graph is the comparability graph  of  a poset intersection of two linear orders if an only if each finite induced graph is a permutation graph (sometimes these graphs are called permutation graphs, while there is no possible permutation involved). For more about permutation graphs, see \cite{klazar}.


\subsubsection{Lexicographical sum}Let $I$ be a poset such that $|I|\geq 2$ and let $\{P_{i}:=(V_i,\leq_i)\}_{i\in I}$ be a family of pairwise disjoint nonempty posets
that are all disjoint from $I$. The \emph{lexicographical sum} $\displaystyle \sum_{i\in I} P_{i}$ is the poset defined on
$\displaystyle \bigcup_{i\in I} V_{i}$ by $x\leq y$ if and only if
\begin{enumerate}[(a)]
\item There exists $i\in I$ such that $x,y\in V_{i}$ and $x\leq_i y$ in $P_{i}$; or
\item There are distinct elements $i,j\in I$ such that $i<j$ in $I$,   $x\in V_{i}$ and $y\in V_{j}$.
\end{enumerate}

The posets $P_{i}$ are called the \emph{components} of the lexicographical sum and the poset $I$ is the \emph{index set}.
If $I$ is a totally ordered set, then $\displaystyle \sum_{i\in I} P_{i}$ is called a \emph{linear sum}. On the other hand, if $I$ is
an antichain, then $\displaystyle \sum_{i\in I} P_{i}$ is called a \emph{direct sum}. Henceforth we will use the symbol $\oplus$ to indicate direct sum.

The decomposition of the incomparability graph of a poset into connected components is expressed in the following lemma which belongs
to the folklore of the theory of ordered sets.

\begin{lemma}\label{lem:folklore} If $P:= (V, \leq)$ is a poset, the order on $P$ induces a total order on the set $Connect(P)$
of connected components of $\inc(P)$, the incomparability graph of $P$, and $P$ is the lexicographical sum of these components indexed by the chain $Connect(P)$. In
particular, if $\preceq$ is a total order extending the order $\leq$ of $P$, each connected component $A$ of $\inc(P)$ is an interval of
the chain $(V, \preceq)$.
\end{lemma}

The next two sections introduce the necessary ingredients to the proof of Theorem \ref{thm:widthtwo}.


\section{A fundamental lemma}\label{section:fund-lemma}

We state an improvement  of I.2.2 Lemme, p.5 of \cite{pouzet78}.
\begin{lemma}\label{lem:inducedpath} Let $x,y$ be two vertices of a poset $P$ with $x<y$.  If $x_0, \dots, x_n$ is an induced path in the incomparability graph of $P$ from $x$ to $y$ then   $x_i< x_j$ for all $j-i\geq 2$.
\end{lemma}
\begin{proof}
Induction on $n$. If $n\leq 2$ the property holds trivially. Suppose $n\geq 3$. Taking  out $x_0$, induction applies to $x_1, \dots x_n$. Similarly, taking out $x_n$, induction applies to $x_0, \dots x_{n-1}$. Since the path from $x_0$ to $x_n$ is induced, $x_0$ is comparable to every $x_j$ with $j\geq 2$ and $x_n$ is comparable to every $x_j$ with $j<n-1$. In particular, since $n\geq 3$, $x_0$ is comparable to $x_{n-1}$. Necessarily, $x_0< x_{n-1}$. Otherwise, $x_{n-1}<x_0$  and then by transitivity $x_{n-1}<x_{n} $ which is impossible since $\{x_{n-1}, x_{n}\} $  is an edge of the incomparability graph. Thus, we may apply induction to the path  from $x_0, \dots,  x_{n-1}$ and  get $x_0<x_j$ for every $j>2$. Similarly, we get $x_1<x_n$ and via the induction applied to the path from $x_1$ to $x_n$, $x_j< x_n$ for $j<n-1$. The stated result follows.
\end{proof}

An immediate corollary is this.

\begin{corollary}\label{cor:cover-distance}Let $P$ be a poset such that $\inc(P)$ is connected and let $a<b$. If $(a,b)$ is a covering relation in $P$, then $2\leq d_{\inc(P)}(a,b)\leq 3$.
\end{corollary}

Another consequence of Lemma \ref{lem:inducedpath} is that incomparability graphs have no induced cycles of length at least five \cite{gallai}. Indeed, let $P$ be a poset and let $x_0,\dots,x_l, x_0$ be an induced cycle of $\inc(P)$. Suppose for a contradiction that $l\geq 4$. We will apply Lemma \ref{lem:inducedpath} successively to the induced paths $x_0,\dots, x_{l-1}$ and $x_1,\dots, x_{l}$ and will derive a contradiction. We may assume without loss of generality that $x_0<x_{l-1}$. It follows from Lemma \ref{lem:inducedpath} applied to $x=x_0$ and $y=x_{l-1}$ that $x_0<x_{l-2}$ (recall that $l\geq 4$) and $x_1<x_{l-1}$. We now consider the induced path $x_1,\dots, x_{l}$. Then $x_1$ and $x_l$ are comparable. It follows from $x_1<x_{l-1}$ and Lemma \ref{lem:inducedpath} applied to $x=x_1$ and $y=x_{l}$ that $x_1<x_l$.  Hence, $x_{l-2}<x_l$. By transitivity we get $x_0<x_l$ which is impossible.

Here is yet another consequence of Lemma \ref{lem:inducedpath}.

\begin{proposition}Let $P:=(V,\leq)$ be a poset. A sequence $a_0,...,a_n,...$ of vertices of $V$ forms an induced path in $\inc(P)$ originating at $a_0$ if and only if for all $i\in \NN$, $a_i,a_{i+1},a_{i+2},a_{i+3}$ is an induced path of $\inc(P)$ with extremities $a_i,a_{i+3}$.
\end{proposition}
\begin{proof}$\Rightarrow$ Obvious. \\
$\Leftarrow$ Suppose that for all $i\in \NN$, $a_i,a_{i+1},a_{i+2},a_{i+3}$ is an induced path with extremities $a_i,a_{i+3}$. We prove by induction that for all $n\in \NN$, $a_0,...,a_n$ is an induced path in $G$. Suppose $a_0,...,a_n$ is an induced path in $G$ and assume without loss of generality that $a_0<a_n$. Then $a_i<a_{n}$ for all $i\leq n-2$ (follows from Lemma \ref{lem:inducedpath}). From $a_{n-2},a_{n-1},a_{n},a_{n+1}$ is an induced path with extremities $a_{n-2},a_{n+1}$ and $a_{n-2}<a_{n}$ we deduce that $a_{n-2}<a_{n+1}$ and $a_{n-1}<a_{n+1}$. Therefore, $a_i<a_{n+1}$ for all $i\leq n-1$ proving that $a_0,...,a_n,a_{n+1}$ is an induced path in $G$.
\end{proof}

We should mention that the value $3$ is the previous proposition is best possible. Indeed, if $P$ the direct sum of two copies of the chain of natural numbers, then $\inc(P)$ is a complete bipartite graph and every path on $3$ vertices is an induced path. Yet an infinite sequence of vertices that alternates between the copies of $\NN$ does not constitute an infinite induced path of $\inc(P)$.

\section{Posets of width $2$ and their distances}\label{posetswidth2}

\subsection{Posets of width $2$ and bipartite permutation graphs}
In this subsection we recall some properties about posets of width at most $2$ and permutation graphs. We start with a characterization of bipartite permutation graphs, next we give some properties of the graphic distance and the detour in comparability graphs of posets of width at most $2$. We recall the existence of a universal poset of width at most $2$ \cite{pouzet78}. We describe the incomparability graph of a variant of this poset more appropriate for our purpose.


\begin{figure}[h!]
\begin{center}
\leavevmode \epsfxsize=2in \epsfbox{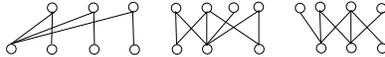}
\end{center}
\caption{Critical posets of dimension $3$ and height $2$.}
\label{fig:critique3}
\end{figure}

We note that a poset $P$ of width at most $2$ has dimension at most $2$, hence its comparability graph is an  incomparability graph. As previously mentioned, a finite graph $G$ is a comparability and incomparability graph if and only if it is a permutation graph. Incomparability graphs of finite posets of width $2$ coincide with bipartite permutation graphs. For arbitrary posets,  the characterization  is  as follows.

\begin{lemma}\label{lem:w2}Let $G$ be a graph. The following are equivalent.
\begin{enumerate}[(i)]
  \item $G$ is  bipartite and is the comparability graph of a poset of dimension at most two;
  \item $G$ is bipartite and embeds  no even cycles of length at least six and none  of the comparability graphs of the posets depicted in Figure $\ref{fig:critique3}$.
  \item $G$ is the incomparability graph of a poset of width at most $2$.
  \item $G$ is a bipartite incomparability graph.
\end{enumerate}
\end{lemma}
\begin{proof}
$(i)\Leftrightarrow (ii)$. If $G$ is finite, this is Theorem 1 of \cite{moore-trotter}. Hence,  the equivalence between $(i)$ and $(ii)$ holds for the   restrictions of $G$ to every finite set $F$ of vertices. This gives immediately the implication $(i)\Rightarrow (ii)$. For the converse implication, we get  that every finite induced subgraph of $G$  is  bipartite and the comparability graph of a poset of dimension at most two.  The Compactness Theorem of First Order Logic implies  that these properties extend to $G$.\\
$(iii)\Rightarrow (i)$. Suppose $G$ is the incomparability graph of a poset of width at most $2$. Then $G$ has no 3-element cycles. Also, $G$ has no induced odd cycles of length at least five (see \cite{gallai}, Section 3.8, Table 5). This shows that $G$ is bipartite. Since $P$ is coverable by two chains it has order dimension two (the dimension of a poset is at most its width \cite{dilworth}) and therefore its incomparability graph is also a comparability graph \cite{dushnik-miller}. Thus $G$ is a comparability graph of a poset of dimension at most two.\\
$(i)\Leftrightarrow (iv)$. Follows from the fact  that a graph $G$ is the incomparability graph of a poset of dimension at most $2$ if and only if this is  the  comparability graph of a poset of dimension at most $2$ \cite{dushnik-miller}.
$(iii) \Leftrightarrow (iv)$. Implication $(iv) \Rightarrow (iii)$ is trivial. For the converse, suppose that $G$ is a bipartite incomparability  graph of a poset $P$, apply Dilworth's theorem\cite{dilworth} or pick   any vertex $x$,  and observe that the set of  vertices at odd distance from $x$ and the set of vertices at even distance from $x$ form a partition of $P$ into two chains, hence $G$ is bipartite.
\end{proof}

We should mention the following result (this is essentially Lemma 14 from \cite{zaguia2008}) which states that a bipartite permutation graph without cycles must embed a caterpillar. A key observation is that if a vertex has at least three neighboring vertices in $\inc(P)$, then at least one has degree one. Otherwise, $\inc(P)$ would have a spider (see Figure \ref{fig:critique3}) as an induced subgraph,  which is impossible.

\begin{lemma}\label{nospider}Let $P$ be a poset of coverable by two chains. Then the following properties
are equivalent.
\begin{enumerate}[(i)]
\item The incomparability graph of $P$ has no cycles of length three or four.
\item The incomparability graph of $P$ has no cycle.
\item The connected components of the incomparability graph of $P$ are caterpillars.
\end{enumerate}
\end{lemma}

\subsection{Detour of bipartite permutation graphs}\label{subsection:detour-perm-bip}

We are going to evaluate the detour of connected components of the incomparability graph of a poset of width at most $2$.

Let $P:=(V,\leq)$ be a poset of width $2$. Suppose that $\inc( P)$ is connected. In this case,  the partition of $P$ into two chains is unique. An \emph{alternating sequence}   in $P$ is any finite monotonic sequence $(x_0, \dots, x_i, \dots x_n)$ of elements of $V$ (i.e., increasing or decreasing) such that no two consecutive elements $x_i$ and $x_{i+1}$ belong to the same chain of the partition. The integer $n$ is the \emph{oscillation} of the sequence; $x$ and $y$ are its \emph{extremities}.

We recall that the oscillation of an alternating sequence with extremities $x$, $y$ is either $0$ or at most $d_{\inc( P)}$ (see I.2.4. Lemme p.6  of \cite{pouzet78}). This allows to define the following map.
Let $d_P$ be the map from $V\times V$ into $\NN$ such that.

\begin{enumerate}
\item $d_P(x,x)= 0$ for every $x\in V$;
\item $d_P(x,y)= 1$ if $x$ and $y$ are incomparable;
\item $d_P(x,y)=2$ if $x$ and $y$ are comparable and there is no alternating sequence from $x$ to $y$;
\item $d_P(x,y)=n+2$ if $n\not =0$ and   $n$ is the maximum of the oscillation of alternating sequences with extremities $x$ and $y$.
\end{enumerate}

We recall a result of \cite{pouzet78} II.2.5 Lemme, p. 6.

\begin{lemma}\label{lem:oscillation-distance}The map $d_P$ is a distance on any poset $P$ of width $2$ such that the incomparability graph is connected.  Moreover,  for every $x,y\in P$ the following inequalities hold:
\begin{equation}
0\leq d_{\inc( P)}(x,y)-d_P(x,y) \leq  2\lfloor d_{\inc( P)}(x,y)/3\rfloor.
\end{equation}
\end{lemma}


We give a slight  improvement of \cite{pouzet78} I.2.3. Corollaire, p. 5.

\begin{lemma}\label{lem:oscillation2}
Let $P$ be poset of width $2$ such that $\inc( P)$ is connected.  Let $n\in \NN$,  $r\in \{0,1\}$ and   $x, y\in P$ such that $\inc( P)$ contains an  induced   path of length $3n+r$ and extremities $x$ and $y$. If  $r\not =1$ and $n\geq 1$(resp. $r=1$ and $n\geq 2$) then there is an alternating sequence with extremities $x$,$y$ and oscillation $n$ (resp. $n-1$).
\end{lemma}

\begin{proof}Since $n\geq 1$, $x$ and $y$ are comparable and we may suppose $x<y$. Let $x_0, \dots, x_{3n+r}$ be a path with $x_0=x$, $x_{3n+r}=y$. According to Lemma \ref{lem:inducedpath} the sequence $x_0, \dots, x_{3i}, \dots x_{3n}$ is alternating. If $r\not =1$, we may replace $x_{3n}$ by $x_{3n+r}$ in the above sequence and get an  alternating sequence with extremities $x$,$y$ and oscillation $n$. If $r=1$, we delete $x_{3n}$ and replace $x_{3(n-1)}$ by $x_{3n+r}$ in the above sequence. We get an alternating sequence of oscillation $n-1$.
\end{proof}

From Lemma \ref{lem:oscillation-distance}, the oscillation between two vertices $x$ and $y$ of $P$  is bounded above. With this lemma, the length of induced paths between $x$ and $y$ is bounded  too, that is the detour $D_{\inc( P)} (x,y)$ is an integer.  In fact we have:

\begin{proposition}
Let $P$ be poset of width $2$ such that $\inc( P)$ is connected and let $x,y\in P$. Then:

\begin{enumerate}[$(1)$]\label{prop:oscillation}
\item $d_{\inc( P)}(x,y)=d_P(x,y)= D_{\inc( P)}(x,y)$  if either $x=y$, in which case this common value is $0$, or $x$ and $y$ are incomparable, in  which case this common value is $1$.
\item $d_{\inc( P)}\geq d_P(x,y)\geq \lfloor  D_{\inc( P)}(x,y)/3 \rfloor +\epsilon$ where $\epsilon=1$ if  $D_{\inc( P)}(x,y)\equiv 1 \mod 3$ and $\epsilon=2$ otherwise.
\end{enumerate}

\end{proposition}

\begin{proof}
Assertion $(1)$ is obvious. For $(2)$, we may suppose $x<y$. The first inequality is embodied in Lemma \ref{lem:oscillation-distance}. As observed above, $D_{\inc( P)}(x,y)$ is bounded. We may write $D_{\inc( P)}(x,y)=3n+r$ with $r$ be  the remainder of $D_{\inc( P)}(x,y) \mod 3$. Let $\alpha:=\lfloor  D_{\inc (P)}(x,y)/3 \rfloor +\epsilon$. We have $\alpha=n+1$ if $r=1$ and $\alpha=n+2$ otherwise. If  $n=0$ then since $x<y$, $r\not =1$, hence $\alpha=2$, since  $d_P(x,y)=2$, the inequality holds. We may suppose $n\geq 1$.  If $r\not =1$ then $\alpha= n+2$, while by definition of $d_P$  and Lemma \ref{lem:oscillation2}, $d_P(x,y) \geq n+2$. Hence, the second inequality holds. If $r=1$  then  $\alpha = n+1$. If  $n=1$ $d_{P}(x,y)\geq 2$ and the second inequality holds. Suppose $n\geq 2$. Then, by definition of $d_{P}(x,y)$ and by  Lemma \ref{lem:oscillation2}, $d_{P}(x,y)\geq n+1$. Thus second inequality holds.
\end{proof}

\begin{corollary}\label{cor:detour} If a bipartite permutation graph has diameter at most $k$ it contains no  induced path of length $3k$.
\end{corollary}

\section{A proof of Theorem \ref{thm:widthtwo}}\label{proof:thm:widthtwo}

\begin{proof}
The implication $(i) \Rightarrow (ii)$ is obvious. The implication $(ii) \Rightarrow (iii)$ follows from Proposition \ref{prop:oscillation} given in Subsection \ref{subsection:detour-perm-bip}. The implication $(iii) \Rightarrow (iv)$ follows from Theorem \ref{thm:infinitepath-kite}. The implication $(iv) \Rightarrow (i)$ is obvious.
\end{proof}

\section{Example \ref{thm:infi-detour-no-path}}\label{proof:thm:infi-detour-no-path}
\begin{proof}Let $X:=\{y,x_0,x_1,x_2,\dots\}$ and for every integer $i\geq 0$ let $Z_{i}:=\{z_{0,i},z_{1,i},\dots,z_{i+3,i}\}$ be disjoint sets. We set $V:=\bigcup_{i\geq 0}Z_i\cup X$ and $P:=(V,\leq)$ where $\leq $ is the binary relation on  $V$ defined as follows: $X\setminus \{y\}$ is totally ordered by $\leq$ and $x_0<x_1<x_2<\dots <x_i<\dots$. For all $0\leq i<j$, every element of $Z_i$ is below every element of $Z_j$. For all $i\geq 0$, $y$ is smaller than all elements in $Z_i$ and is incomparable to $x_i$. For all $i\geq 0$, $x_i$ is smaller than all element of $Z_i\setminus \{z_{0,i}\}$ and $x_i$ is incomparable to all elements in $\bigcup_{j<i}Z_i\cup \{z_{0,i}\}$. For all integers $i\geq 0$ and for all $j\geq i+1$, $x_i$ is smaller than all element in $Z_j$. Finally, the restriction of $\inc(P)$ to $Z_i$ is the induced path $z_{0,i},z_{1,i},\dots,z_{i+3,i}$ so that $z_{0,i}<z_{2,i}<z_{4,i}<\dots$ and $z_{1,i}<z_{3,i}<z_{5,i}<\dots$ (see Figure \ref{width-three}). It is not difficult to see that $\leq$ is an order relation and that the corresponding poset $P$ can be covered by three chains.\\
\textbf{Claim 1:} The diameter of $\inc(P)$ is $3$.\\
Let $a,b$ be two distinct vertices of $\inc(P)$. If $a,b\in X$, then either $a=y$ or $b=y$ in which case $d_{\inc(P)}(a,b)=1$, or $y\not \in \{a,b\}$ in which case $d_{\inc(P)}(a,b)=2$ (indeed, say $a=x_i$ and $b=x_j$ with $i<j$, then $a,z_{0,i},b$ is an induced path in $\inc(P)$). Suppose now $a\in X$ and $b\not \in X$, say $b\in Z_i$ for some $i\geq 0$. If $a=y$, then $d_{\inc(P)}(a,b)=2$ (indeed, $a,x_{i+1},b$ is an induced path in $\inc(P)$). Else if $a=x_j$ for some $j\geq 0$, then  $d_{\inc(P)}(a,b)=1$ if $i<j$ and $d_{\inc(P)}(a,b)=3$ otherwise (indeed, $a,z_{0,j},x_{i+1},b$ is the shortest path joining $a$ to $b$). Next we suppose that $\{a,b\}\cap X=\varnothing$. If $a,b\in Z_i$ for some $i\geq 0$, then $d_{\inc(P)}(a,b)=2$ (indeed, $a,x_{i+1},b$ is an induced path in $\inc(P)$). Else if $a\in Z_i$ and $b\in Z_j$  for some $i\neq j$, then $d_{\inc(P)}(a,b)=2$ (indeed, $a,x_{i+j},b$ is an induced path in $\inc(P)$).\\
\textbf{Claim 2:} An induced infinite path in $\inc(P)$ contains necessarily finitely many elements of $X$.\\
Suppose an induced infinite path $C$ contains infinitely many vertices from $X$. Since $\inc(P)$ induces an independent set on $X\setminus \{y\}$ and $C$ is connected we infer that $C$ must meet infinitely many $Z_i$'s. Hence, there exists some $x_i\in C$ which has degree at least $3$ in $C$ and this is not possible.\\
\textbf{Claim 3:} Deleting all vertices of $X$ from $\inc(P)$ leaves a disconnected graph.\\
Clearly, for all $i\geq 0$, $Z_i$ is a connected component of $\inc(P)\setminus X$.\\
\noindent Now suppose for a contradiction that $\inc(P)$ embeds an infinite induced path $C$. It follows from Claim 2 that we can assume $V(C)\cap X=\varnothing$. Hence, $C$ is an induced infinite path of $\inc(P)\setminus X$. We derive a contradiction since all connected components of $\inc(P)\setminus X$ are finite (indeed, the connected components of $\inc(P)\setminus X$ are finite paths i.e. the subgraphs of $\inc(P)\setminus X$ induced on the $Z_i$'s).

\noindent \textbf{Claim 4:} The vertex $y$ has an infinite induced detour.\\
Indeed, $\inc(P)$ induces a path on $\{y,x_i,\}\cup Z_i$ of length $i+5$ for all $i\geq 0$.
\end{proof}

\section{Order and metric convexities of incomparability graphs}\label{section:convexity}
In this section we compare the notions of order convexity and metric convexity with respect to the distance on the incomparability  graph of a poset.

We recall few definitions already provided in the introduction. Let $P:=(V,\leq)$ be a poset. We recall $Conv_P(X)$ is the smallest convex set containing $X$ and that
\[Conv_P(X)=\{z\in P : x\leq z\leq y \mbox{ for some } x,y\in X\}=\downarrow X\cap \uparrow X.\]

Let $G:=(V,E)$ be a graph. We equip it with the graphic distance $d_G$. A \emph{ball} is any subset $B_G(x, r):= \{y\in V: d_G(x,y)\leq r\}$ where $x\in V, r\in \NN$. A subset of $V$ is \emph{convex} with respect to the distance $d_G$ if this is an intersection of balls. The \emph{least convex subset} of $G$ containing $X$ is
\[Conv_{G}(X):=\displaystyle \bigcap_{X\subseteq B_G(x,r)}B_G(x,r).\]

Let $X\subseteq V$ and $r\in \NN$. Define
\[B_G(X,r):=\{v\in V : d_G(v,x)\leq r \mbox{ for some } x\in X\}.\]


The proof of the following lemma is elementary and is left to the reader.

\begin{lemma}\label{lem:b_g}Let $G$ be a graph, $X\subseteq V(G)$ and $r\in \NN$. Then
\begin{enumerate}[$(1)$]
  \item $B_G(X,r)=B_G(B_G(X,1),r-1)= B_G(B_G(X,r-1), 1)$ for all $r\geq 1$.
  \item $B_G(X\cup Y,r)=B_G(X,r)\cup B_G(X,r)$.
\end{enumerate}
\end{lemma}


\begin{lemma}\label{lem:convex-boule}Let $P:=(V,\leq)$ be a poset and $G$ be its incomparability graph, $X\subseteq V$ and $r\in \NN$. Then
\begin{equation}\label{eq1}
B_G(\downarrow X,r)=(\downarrow X)\cup B_G(X,r)=\downarrow B_G(X,r).
\end{equation}
\begin{equation}\label{eq2}
B_G(\uparrow X,r)=(\uparrow X)\cup B_G(X,r)=\uparrow B_G(X,r).
\end{equation}
\begin{equation}\label{eq3}
B_G(\uparrow X\cap \downarrow X,r)= B_G(\uparrow X,r)\cap B_G(\downarrow X,r).
\end{equation}
\begin{equation}\label{eq4}
B_G(Conv_P(X),r)=Conv_P(X)\cup B_G(X,r)=Conv_P(B_G(X,r)).
\end{equation}
\end{lemma}
\begin{proof}We mention at first that all above equalities are clearly true for $r=0$. We claim that it is enough to prove (\ref{eq1}). Indeed, (\ref{eq2}) is obtained from (\ref{eq1}) applied to $P^{*}$. We now show how to obtain (\ref{eq3}) using (\ref{eq1}) and (\ref{eq2}). The proof is by induction on $r$. \\
Basis step: $r=1$.\\
Clearly, $B_G(\uparrow X\cap \downarrow X,1)\subseteq B_G(\uparrow X,1)\cap B_G(\downarrow X,1).$ Let $x\in B_G(\uparrow X,1)\cap B_G(\downarrow X,1)$. There are $y_1\in \downarrow X$ and $y_2\in \uparrow X$  such that $x$ is equal to $y_1$ or incomparable to $y_1$ and similarly $x$ is equal to $y_2$ or incomparable to $y_2$. Since $y_1\in \downarrow X$ and $y_2\in \uparrow X$ there are $x_1, x_2\in X$ such that  $y_1\leq x_1$ and $x_2\leq y_2$. If $x$ is incomparable or equal to $x_1$ or to $x_2$, then $x\in B_G(X, 1)\subseteq B_G(\uparrow X\cap \downarrow X,1)$ as required. If not, $x_2\leq x\leq x_1$ (since $x$ is equal to $y_1$ or incomparable to $y_1$ and  $x$ is equal to $y_2$ or incomparable to $y_2$), hence $x\in \downarrow X\cap \uparrow X\subseteq B_G(\downarrow X\cap \uparrow X, 1)$, as required.

%
Inductive step: Suppose $r>1$. We have
\begin{eqnarray*}
B_G(\uparrow X\cap \downarrow X,r) &=& B_G(B_G(\uparrow X\cap \downarrow X,r-1),1)\\
&=& B_G(B_G(\uparrow X,r-1)\cap B_G(\downarrow X,r-1),1)\;  \mbox{(by the induction hypothesis})\\
&=& B_G(\uparrow B_G(X,r-1)\cap \downarrow B_G(X,r-1),1)\;  \mbox{(by equations (\ref{eq1}) and (\ref{eq2})})\\
&=& B_G(\uparrow B_G(X,r-1),1)\cap B_G(\downarrow B_G(X,r-1),1)\;  \mbox{(follows from the basis step $r=1$})\\
&=& \uparrow B_G( B_G(X,r-1),1)\cap \downarrow B_G( B_G(X,r-1),1)\;  \mbox{(follows from (\ref{eq1}) and (\ref{eq2})} )\\
&=& \uparrow B_G(X,r)\cap \downarrow(B_G(X,r))\\
&=& B_G(\uparrow X,r)\cap (\downarrow B_G(X,r))\;  \mbox{(follows from (\ref{eq1})} ).
\end{eqnarray*}

We now show how to obtain (\ref{eq4}) using (\ref{eq1}), (\ref{eq2}) and (\ref{eq3}).\\
From (\ref{eq1}) and (\ref{eq2}) we obtain
\[B_G(\downarrow X,r)\cap B_G(\uparrow X,r) =((\downarrow X)\cup B_G(X,r))\cap((\uparrow X)\cup B_G(X,r))=\downarrow(B_G(X,r))\cap \uparrow(B_G(X,r)).\]
This is equivalent to
\[B_G(\downarrow X,r)\cap B_G(\uparrow X,r) =(\downarrow X\cap \uparrow X)\cup B_G(X,r))=\downarrow(B_G(X,r))\cap \uparrow(B_G(X,r)).\]
Using (\ref{eq3}) we have
\[B_G(\downarrow X\cap \uparrow X,r) =(\downarrow X\cap \uparrow X)\cup B_G(X,r))=\downarrow(B_G(X,r))\cap \uparrow(B_G(X,r)).\]
The required equalities follow by definition of the operator $Conv$.

We now prove (\ref{eq1}).\\
Basis step: $r=1$.\\
Since $X\subseteq \downarrow X$ we have $B_G(X,1)\subseteq B_G(\downarrow X,1)$. Hence, we have $B_G(\downarrow X,1)\supseteq (\downarrow X)\cup B_G(X, 1)$. From $X\subseteq B_G(X,1)$ we deduce that $\downarrow X\subseteq \downarrow(B_G(X, 1))$. Hence, $(\downarrow X)\cup B_G(X,1)\subseteq \downarrow B_G(X,1)$.

Next, we prove that $B_G(\downarrow X,1)\subseteq (\downarrow X)\cup B_G(X,1)$. Let $x\in B_G(\downarrow X,1)$. There exists then $y\in \downarrow X$ at distance at most $1$ from $x$ that is either $y=x$ or $y\parallel x$. If $y= x$ then $x\in \downarrow X$. Otherwise, since $y\in \downarrow X$  there is $y_1\in X$ such that $y\leq y_1$. If $y_1$ is incomparable or equal to $x$ then $x\in B_G(X,1)$. Otherwise $y_1$ is comparable to $x$. Necessarily, $x \leq y_1$ since  $x\parallel y$.  Hence $x\in \downarrow X$.\\
Inductive step: Let $r>1$.  We suppose true the equalities
\[B_G(\downarrow X,r-1)=(\downarrow X)\cup B_G(X,r-1)=\downarrow B_G(X,r-1).\]
We apply the operator $T\longrightarrow B_G(T,1)$ to each term of the previous equalities and obtain
\[B_G(B_G(\downarrow X,r-1),1)=B_G((\downarrow X)\cup B_G(X,r-1),1)=B_G(\downarrow B_G(X,r-1),1).\]
We have 
\[B_G(B_G(\downarrow X,r-1),1)= B_G(\downarrow X,r) \mbox{ (see (1) of Lemma \ref{lem:b_g})}.\]
Also,
\begin{eqnarray*}
B_G((\downarrow X)\cup B_G(X,r-1),1) &=& B_G(\downarrow X, 1)\cup B_G(B_G(X,r-1),1)\mbox{ (see (2) of Lemma \ref{lem:b_g}))}\\
&=& B_G(\downarrow X,1)\cup B_G(X,r) \mbox{ (see (2) of Lemma \ref{lem:b_g})}\\
&=& (\downarrow X) \cup B_G(X,1) \cup B_G(X,r) \mbox{ (follows from (\ref{eq1}) with $r=1$)}.\\
&=& (\downarrow X)\cup B_G(X,r).
\end{eqnarray*}
Finally we have
\begin{eqnarray*}
B_G(\downarrow(B_G(X,r-1)),1) &=& \downarrow B_G((B_G(X,r-1)),1)\\
&=& \downarrow (B_G(X,r)).
\end{eqnarray*}
\end{proof}

\section{A proof of Theorem \ref{thm:orderconvex} and some consequences}\label{section:proof-thm-orderconvex}
We now proceed to the proof of Theorem \ref{thm:orderconvex}.

\begin{proof}
$(a)$ Apply  successively equations (\ref{eq1}), (\ref{eq2}) and (\ref{eq4}) of Lemma \ref {lem:convex-boule}.

$(b)$ Suppose $r=1$. Let  $G': =G_{\restriction B_G(X, 1)}$ and $x,y\in B_G(X, 1)$. Let $n:=d_G(x,y)$. Clearly, $n\leq d_{G'}(x,y)$. To prove that the equality holds,  we may suppose that $2\leq n<\infty$. We argue by induction on $n$. Let  $u_0, \dots, u_n$ be a path in $G$ connecting  $x$ and $y$. If $n\geq 4$, we have $x_0< x_2<x_n$ by Lemma \ref{lem:inducedpath}. Since $B_G(X, 1)$ is convex, it contains $x_2$, hence, by induction,
$d_G(x,x_2)= d_{G'}(x,x_2)=2$ and $d_G(x_2,y)= d_{G'}(x_2, y)=n-2$, hence $d_G(x,y)= d_{G'}(x,y)$. Thus, to conclude,  it suffices to solve the cases $n=2$ and $n=3$. Let $x',y'\in X$ with $x'$ incomparable or equal to  $x$ and $y'$ incomparable or equal to $y$.
If $u_{n-1}$ is incomparable or equal to $y'$ then $x_{n-1}\in B_G(X, 1)$. From the induction, $d_{G'}(x, x_{n-1})= d_{G}(x, x_{n-1})$ hence
$d_{G'}(x, y)= d_{G}(x, y)$ as required. Hence,  we may suppose $u_{n-1}$ comparable to $y'$, and similarly  $u_{1}$ comparable  to $x'$.

Also, if $x'$ is incomparable or equal to $x_2$ then $x, x', x_2$ is a path in $B_G(X)$; if $n=2$ we have $d_{G'}(x,y)=2$ as required, if $n=3$, then $x, x', x_2, y$  is a path in $B_G(X)$ and  $d_{G'}(x,y)=3$ as required. Thus we may suppose  $x'$ comparable to $x_2$ and, similarly, $y'$ comparable to $x_{n-2}$.

Since $x'$ is incomparable or equal to $x_0$ and, by Lemma \ref{lem:inducedpath}, $x_0<x_2$, we have $x'< x_2$. Similarly,  we have $x_{n-2}<y'$.
Since $x_1$ is comparable to $x'$ and  incomparable to $x_2$ we deduce $x' \leq x_1$ from $x'<x_2$. Similarly, we deduce $x_{n-1}\leq y'$. For $n=2$ we have  $x'\leq x_1, \leq y'$ and for $n=3$,  $x'\leq x_1, x_2\leq y'$. By order convexity of $X$, $x_1\in X$ (and also $x_2\in X$  if $n=3$, hence the path $x=x_0,x_1,x_2=y$ if $n=2$ and the path $x=x_0,x_1,x_2, x_3=y$ if $n=3$ is in $B_{G}(X)$  and thus $d_{G'}(x,y)=n$.

Suppose  $r>1$. Then from $(a)$ above,  $B_G(X, 1)$ order convex. Via the induction hypothesis,   $G_{\restriction  B_G(B_G(X, 1), r-1)}$ is an isometric subgraph of $G$.  Since  $B_G(X, r)=B_G(B_G(X, 1), r-1)$, $G_{\restriction B_G(X,r)}$ is an isometric subgraph of $G$.
\end{proof}

As the proof of the Lemma \ref{lem:convex-boule} suggests, balls are not necessarily geodesically convex (for an example, look at the  ball $B(x, 1)$ in a four element cycle). A consequence of Theorem \ref{thm:orderconvex} is that the order convexity of balls is equivalent to the following inequality:

\begin{corollary}\label{lem:monotone-distance}Let $P$ be a poset and let $G$ be its incomparability graph. Then \begin{equation} \label{eq:metric-inequalities}
d_{G}(u,v)\leq d_{G}(x,y) \mbox{ for all }  x\leq u\leq v\leq y  \mbox{ in }  P.
\end{equation}
\end{corollary}
\begin{proof}
The inequality above amounts to $d_{G}(u,v)\leq d_{G}(x,v)\leq d_{G}(x,y)$. We prove the first inequality;  the second inequality follows by the same argument applied to the dual of $P$.  We may suppose that $x<u<v$,  otherwise nothing to prove. Let  $n:=d_{G}(v,x)$.  By  $(a)$ of Theorem  \ref{thm:orderconvex}, $B_G(v, n)$ is order convex. Since $x, v\in B_G(v, n)$ and $x\leq u\leq v$, then $u\in B(v, n)$ amounting to $d_G(u, v)\leq n=d_{G}(x,v)$. Conversely, assuming that inequality (\ref{eq:metric-inequalities}) holds, observe that every ball $B_G(x,r)$ is order-convex. We may suppose $r\geq 1$, otherwise the conclusion is obvious. Let $u,v\in B_G(x,r)$ and $w\in P$ with $u<w<v$.  If $x \parallel w$,  then $d_G(x, w)=1\leq r$ hence $w\in B_G(x,r)$. If not, then either $x<w$ or $w<x$.  In the first case,  from $x<w<v$, inequality (\ref{eq:metric-inequalities}) yields  $d_G(x, w)\leq d_G(x, v)\leq r$ hence $w\in B_G(x,r)$, whereas in the second case, from  $u<w<x$, inequality (\ref{eq:metric-inequalities}) yields  $d_G(w, x)\leq d_G(u, x)\leq r$ hence $w\in B_G(x,r)$.
\end{proof}

\begin{corollary}\label{lem:intermediateinducedpath} $\delta_G(X)= \delta_G(Conv_{P}(X))=\delta_G(Conv_{G}(X))$ for every subset $X$ of a poset $P$.
\end{corollary}
\begin{proof} Since by $(a)$ of Theorem \ref{thm:orderconvex}, each ball $B_G(x,r)$ is order convex, $Conv_{P}(X) \subseteq Conv_{G}(X)$. Hence  $\delta_G(X)\leq \delta_G(Conv_{P}(X))\leq \delta_G(Conv_{G}(X))$.

The equality $\delta_G(X)=\delta_G(Conv_{G}(X))$ is a general convexity property of metric spaces. Let $r:= \delta_G(X)$. Let $x,y \in Conv_G(X)$. We prove that $d_G(x,y) \leq r$.
First $X\subseteq B_G(x,r)$.  Indeed, let $z\in X$; since $\delta_G(X)=r$, $X \subseteq B_G(x, r)$. Since $Conv_G(X)$ is the intersections of balls containing $X$, we have $Conv(X)  \subseteq B_G(z,r)$, hence $z\in B_G(x, r)$.
Next, from  $X\subseteq B_G(x,r)$ we deduce $Conv(X) \subseteq B_G(x,r)$ hence $y\in B_G(x,r)$ that is $d_G(x,y)\leq r$.
\end{proof}

\begin{lemma}\label{lem:inequality}Let $P:=(V,\leq)$ be a poset and $G$ be its incomparability graph. Let $x,y,z\in V$ be such that $x<z<y$. Then
\[\max\{d_G(x,z),d_G(z,y)\}\leq d_G(x,y)\leq d_G(x,z)+d_G(z,y)\leq d_G(x,y) +2.\]
\end{lemma}
\begin{proof}The first inequality follows from Corollary \ref{lem:monotone-distance}. The second inequality is the triangular inequality. We now prove the third inequality. Let $p:=d_G(x,z)$, $q:=d_G(z,y)$, $r:=d_G(x,y)$. \\
\textbf{Claim:} Let $x_0:=x,\dots, x_r:=y$ be a path from $x$ to $y$. Then there exist $i\not \in \{0,r\}$ such that $z$ is incomparable to $x_i$.\\
\noindent{\bf Proof the claim.} By induction on $r$. Note that since $x<y$ we have $r\geq 2$. If $r=2$, then necessarily $z$ is incomparable to $x_1$. Suppose $r>2$. Then $z\nleq x_1$. If $z$ is incomparable to $x_1$, then we are done. Otherwise $x_1<z$ and we may apply the induction hypothesis to $x_1,y$ and the path $x_1,\dots, x_r=y$. This completes the proof of the claim.\hfill $\Box$

Let $i$ be such in the Claim. Then $x_0:=x,\dots, x_i, z$ is a path from $x$ to $z$ of length $i+1$ and $z, x_i, x_{i+1}, \dots, x_r$ is a path from $z$ to $y$ of length $r-i+1$. Then $p+q\leq i+1+r-i+1=r+2$. The proof of the lemma is now complete.
\end{proof}

\begin{lemma}Let $x_0,...,x_n$ be an isometric path in a graph $G$ with $n\geq 2$. There exists a vertex $x_{n+1}$ such that $x_0,...,x_n,x_{n+1}$ is an isometric path in $G$ if and only if $B_G(x_n,1)\nsubseteq B_G(x_0,n)$.
\end{lemma}
\begin{proof}$\Rightarrow$ is obvious.\\
$\Leftarrow$ Suppose $B_G(x_n,1)\nsubseteq B_G(x_0,n)$ and let $x_{n+1} \in B_G(x_n,1)\setminus B_G(x_0,n)$. \\
\textbf{Claim 1:} $d_G(x_0,x_{n+1})=n+1$.\\
Indeed, since $x_{n+1} \in B_G(x_n,1)\setminus B_G(x_0,n)$ we have $d_G(x_0,x_{n+1})>n$. From the triangular inequality $d_G(x_0,x_{n+1})\leq d_G(x_0,x_{n})+ d_G(x_n,x_{n+1})=n+1$.\\
\textbf{Claim 2:} $d_G(x_{j},x_{n+1})=n+1-j$ for all $0\leq j\leq n$.\\
Indeed, From the triangular inequality $d_G(x_{j},x_{n+1})\leq d_G(x_{j},x_{n})+ d_G(x_{n},x_{n+1})=n-j+1$. Similarly, $d_G(x_{0},x_{n+1})\leq d_G(x_{0},x_{j})+ d_G(x_{j},x_{n+1})$ and therefore $d_G(x_{j},x_{n+1})\geq d_G(x_{0},x_{n+1})-d_G(x_{0},x_{j})=n+1-j$. The equality follows.
\end{proof}

We could restate the previous lemma as follows. There is an isometric path of length $n+1$ starting at some vertex $x_0$ if there is some $x_n\in B_G(x_0,n)$ such that $B_G(x_n,1)\nsubseteq B_G(x_0,n)$.

An other consequence of the convexity of balls in an incomparability graph is the following:

\begin{lemma} \label{ball-infinitepath}Let $G$ be the incomparability graph of a poset $P$. If a ball contains infinitely many vertices of a one way infinite induced path then it contains all vertices except may be  finitely many vertices of that path.
\end{lemma}
\begin{proof} Let $\mathrm P_{\infty}$ be an infinite induced path of $G$ and $(x_n)_{n\in \NN}$ be an enumeration of its vertices, so that $(x_n, x_{n+1}) \in E(G)$ for $n\in \NN$. Without loss of generality we may suppose that $x_0<x_2$ (otherwise,  replace the  order of $P$ by its dual). By Lemma \ref {lem:inducedpath} we have $x_i< x_j$ for every  $i+2\leq j$. Let $B_G(x, r)$ be a ball of $G$ containing infinitely many vertices of $\mathrm P_{\infty}$. Let $x_i\in \mathrm P_{\infty} \cap B(x,r)$. We claim that $x_j\in  \mathrm P_{\infty} \cap B(x,r)$ for all $j\geq i+2$. Indeed, due to our hypothesis, we may pick $x_r \in  \mathrm P_{\infty} \cap B(x,r)$ with $r\geq j+2$. We have $x_i<x_j<x_r$. Due to the convexity of $B(x,r)$ we have $x_j\in B(x, r)$. This proves our claim.
\end{proof}

Said differently:
\begin{lemma}\label{ball-infinitepath2}
If a one way infinite  induced path $\mathrm P_{\infty}$ has an infinite diameter in the incomparability graph $G$ of a poset then every ball of $G$ with finite radius contains only finitely many vertices of $\mathrm P_{\infty}$.
\end{lemma}



\section{Induced infinite paths in incomparability graphs: A proof of Theorem \ref{thm:infinitepath-kite}}\label{section:proof-thm:infinitepath-kite}

The proofs of $(1)$ and $(2)$ of Theorem \ref{thm:infinitepath-kite} are similar. We construct a  strictly increasing sequence $(y_n)_{n\in \NN}$ of vertices such that $3\leq d_G(y_n,y_{n+1})<+\infty$ for all $n\in \NN$  and we associate to each $n\in \NN$  a finite path  $\mathrm P _n:=z_{(n,0)},z_{(n,1)},...,z_{(n,r_n)}$ of $G$ of length $r_n:=d_G(y_n,y_{n+1})$ joining $y_n$ and $y_{n+1}$. We show first that the graph  $G':=G_{\restriction \bigcup_{n\in \NN}V(\mathrm P_n)}$  is connected and has an infinite diameter. Next, we prove  that it  is locally finite. Hence from K\H{o}nig's Lemma (\ref{thn:konig}), it contains an isometric path. This path yields  an induced path of $G$. The detour via K\H{o}nig's Lemma is because the union of the two consecutive paths   $\mathrm P_n$ and $\mathrm P_{n+1}$ do not form necessarily a path. In the first proof, our paths have length $3$. In the second proof, their end vertices have degree at least $3$.

\begin{lemma}\label{lem:oneside}Let $P:=(V,\leq)$ be a poset so that its incomparability graph $G$ is connected and has infinite diameter. Let $x\in V$ be arbitrary. Then at least one of the sets $d_G^+(x):=\{d_G(x, y) :  x<y\in V\}$ or $d_G^-(x):=\{d_G(x, y): y\in V \;\text{and}\; y<x\}$ is unbounded in $\NN$. Furthermore, if $d_G^+(x):=\{d_G(x, y) :  x<y\in V\}$ is unbounded in $\NN$ and $z>x$, then $d_G^+(z):=\{d_G(z, y) :  z<y\in V \}$ is unbounded in $\NN$ (in particular, $z$ cannot be maximal in $P$).
\end{lemma}
\begin{proof}Suppose for a contradiction that the sets $d_G^+(x):=\{d_G(x, y) :  x<y\in V\}$ and $d_G^-(x):=\{d_G(x, y): y\in V \;\text{and}\; y<x\}$ are bounded. Let $r:=\max d_G^+(x)$ and $r':=\max d_G^-(x)$. Then $V:=B_G(x,\max\{2, r,r'\})$ and therefore the diameter of $G$ is bounded contradicting our assumption. Now let $z>x$ and suppose for a contradiction that $d_G^+(z):=\{d_G(z, y) :  z<y\in V \}$ is bounded and let $r:=\max d_G^+(z)$. Let $x<y$. If $y\leq z$ then $d(x,y)\leq d(x, z)$ by Lemma \ref{lem:inequality};  if $z\parallel y$ then  $d(x,y)\leq d(x,z)+1$;  if $z\leq y$ then we have $d(x, y)\leq d(x,z)+d(z,y)\leq d(x,z)+r$, hence, the set $d_G^+(x)$ is bounded,  contradicting our assumption.
\end{proof}

\noindent {\bf Proof of (1) of Theorem \ref{thm:infinitepath-kite}.}
We construct a sequence $(x_{n})_{n\in \NN}$ of vertices (see Figure \ref{fig:sequence}). We pick  $x_0\in V$. According to Lemma \ref{lem:oneside}, one of the set $d_G^+(x_0):=\{d_G(x_0, y) :  x_0<y\in V\}$ and $d_G^-(x_0):=\{d_G(x_0, y): y\in V \;\text{and}\; y<x_0\}$ is unbounded. We may assume without loss of generality that the set $d_G^+(x_{0} )$ is unbounded. Choose an element $x_3>x_0$ at distance three from $x_0$ in $G$ and let $x_{0}, x_{1}, x_{2}, x_{3}$ be a path joining $x_0$ to $x_1$. Note that necessarily we have $x_0<x_{2}$ and $x_{1}<x_3$. Now suppose constructed a sequence $x_0, x_1, \dots, x_{3n}$ such that $x_0< x_3\dots <x_{3n}$ and such that $x_{3i}, x_{3i+1}, x_{3i+2}, x_{3i+3}$ is a path of extremities $x_{3i}$ and $x_{3(i+1)}$ for $i<n$. According to Lemma \ref{lem:oneside}, the  set $d_G^+(x_{3n})$ is unbounded. Hence, it contains a vertex $x_{3(n+1)}$ at distance three from $x_{3n}$. Let  $x_{3n}, x_{3n+1}, x_{3n+2}, y_{3n+3}$ be a path of extremities $x_{3n}$ and $x_{3(n+1)}$. By Lemma \ref{lem:inducedpath} we have necessarily:
\begin{equation}\label{equ:inequality inducedpath}
x_{3n}<x_{3n+2}\;  \text{and}\;  x_{3n+ 1}<x_{3n+3}.
\end{equation}


\begin{figure}[h!]
\begin{center}
\leavevmode \epsfxsize=1.2in \epsfbox{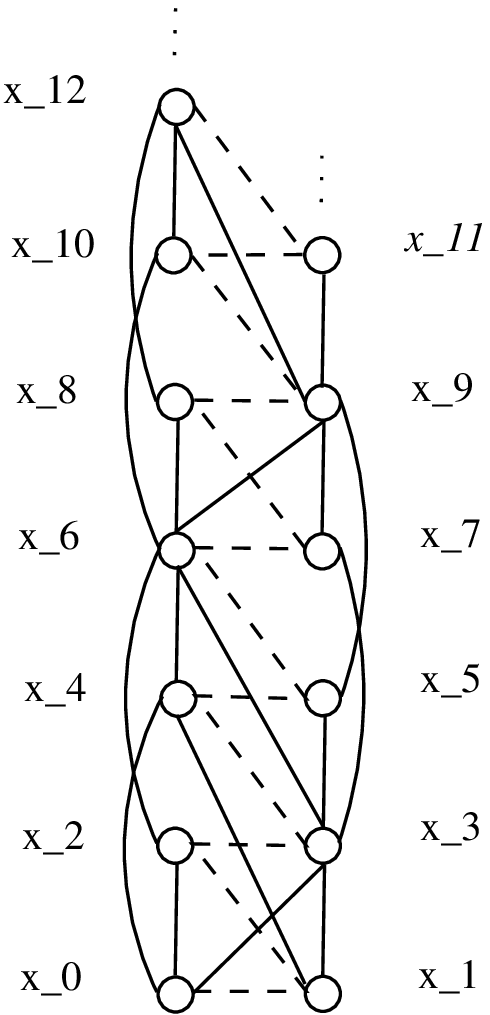}
\end{center}
\caption{}
\label{fig:sequence}
\end{figure}

Let $P'$ be the poset induced on the set $V':= \{x_n :n\in \NN\}$ and $G'$ be the incomparability graph of $P'$. According to our construction $G'$ contains  a spanning path (not necessarily induced), hence it is connected.

\begin{claim}\label{claim:inequalityG}  $d_{G}(x_0,x_{3n})\geq n+2$ for every $n\geq 1$.
\end{claim}
Since $d_{G'}(x_0,x_{3n})\geq d_{G}(x_0,x_{3n)})$, it follows that the  diameter of $G'$ is infinite.

\noindent{\bf Proof of  Claim \ref {claim:inequalityG}.}
We prove the inequality of the claim by induction on $n\geq 1$. By definition, the inequality holds for $n=1$.  Suppose the inequality holds for $n$. It follows from  Lemma \ref{lem:inequality} that $n+5 \leq d_G(x_0, x_{3n})+ d_G(x_{3n}, x_{3(n+1)})\leq d_G(x_0, x_{3(n+1)})+2$ and therefore the inequality holds for $n+1$.

\begin{claim}\label{claim:pathkonig} The incomparability graph of $P'$ is locally finite, that is for all $x\in P'$, $\iinc_{P'}(x):= \{y\in V': x \parallel y\}$ is finite.
\end{claim}In fact, $\iinc_{P'}(x)$ has at most six elements.

\noindent{\bf Proof of  Claim \ref {claim:pathkonig}.} We have
\begin{enumerate}[$(a)$]
  \item $\iinc_{P'}(x_{3n})\subseteq \{x_{3n-1}, x_{3n+1}\}$ for $n\geq 1$.
  \item $\iinc_{P'}(x_{3n+1})\subseteq \{x_{3(n-2)+2}, x_{3(n-1)+1}, x_{3(n-1)+2},  x_{3n},  x_{3n+2}, x_{3(n+1)+1}\}$ for $n\geq 2$.
  \item $\iinc_{P'}(x_{3n+2})\subseteq \{x_{3n-1}, x_{3(n+1)}, x_{3(n+1)+1} x_{3(n+1)+2}, x_{3(n+2)+1}\}$ for $n\geq 1$.
  \item  $\iinc_{P'}(x_{0})=\{x_1\}$, $\iinc_{P'}(x_{1})\subseteq \{x_0,x_2,x_4\}$, $\iinc_{P'}(x_2)\subseteq \{x_1,x_3,x_5,x_7\}$ and $\iinc_{P'} (x_4)\subseteq \{x_1,x_3,x_5,x_7\}$.
\end{enumerate}
\begin{proof}
\begin{enumerate}[$(a)$]
  \item Let $n\in \NN$. By inequalities (\ref{equ:inequality inducedpath}) stated above,  we have $x_{3n-2}<x_{3n}< x_{3n+2}$. Let $n'\in \NN$ be such that $n<n'$. By construction, $x_{3n}<x_{3n'}$. By inequalities (\ref{equ:inequality inducedpath}) again  we have $x_{3n'}<x_{3n'+2}$, hence  $x_{3n}<x_{3n'+2}$. Since $x_{3n}<x_{3n'}$ and $x_{3n'+1}$ is incomparable to $x_{3n'}$ we infer that $x_{3n'+1}\nleq x_{3n}$. We have  $d_G(x_{3n},x_{3n'})\geq 3$; indeed, if $n'=n+1$, $d_G(x_{3n},x_{3n'})=3$ by construction,  otherwise apply the first inequality of Lemma \ref{lem:inequality} with $x=x_{3n}$, $z=x_{3(n+1)}$ and $y=x_{3n'}$).  Since  $d_G(x_{3n},x_{3n'})\geq 3$ and $x_3$ is incomparable to $x_{3n+1}$, the vertices $x_{3n}$ and $x_{3n'+1}$ cannot be incomparable; it follows that $x_{3n}<x_{3n'+1}$.

Since  a poset and its dual have the same incomparability graph, we deduce that if $n'<n$, then $x_{3n'},x_{3n'+1}, x_{3n'+2}<x_{3n}$. Hence, $\iinc_{P'}(x_{3n})\subseteq \{x_{3n-1}, x_{3n+1}\}$ for $n\geq 1$.

        \item Since $x_{3n-3}<x_{3n}$ and $x_{3n}$ and $x_{3n+1}$ are incomparable we infer that $x_{3n+1}\nleqslant x_{3n-3}$. It follows that $x_{3n-3}<x_{3n+1}$ because otherwise $x_{3n-3},x_{3n+1},x_{3n}$ would be a path of length two contradicting our assumption that $d_G(x_{3n-3},x_{3n})=3$. From Lemma \ref{lem:inducedpath},  we deduce that if $k<3n-4$, then $x_k<x_{3n-3}$ and hence $x_k<x_{3n+1}$. Hence, if $k<3n-1$ and $x_k$ is incomparable to $x_{3n+1}$, then $k\in \{3n-4,3n-2\}$. Since $x_{3n+1}<x_{3n+3}$ it follows from Lemma \ref{lem:inducedpath} that if $k>3n+4$, then $x_k>x_{3n+4}$ and hence $x_k\not \in \iinc_{P'}(x_{3n+1})$. Hence, $x_{3n}$ and   $x_{3n+1}$ are possible elements incomparable to $x_{3n+1}$, hence the required inclusion.
    \item Since $x_{3n}<x_{3n+2}$ it follows from Lemma \ref{lem:inducedpath} that if $x_k$, for $k<3n$, is incomparable to $x_{3n+2}$ then $k\in \{3n-1,3n+3\}$. Now observe that $x_{3n+2}<x_{3n+6}$ because otherwise $x_{3n+3},x_{3n+2},x_{3n+6}$ is a path of length two contradicting $d_G(x_{3n+3},x_{3n+6})=3$. By duality we infer that if $k>3n+4$, then $x_k$ incomparable to $x_{3n+2}$ implies $k\in \{3n+5,3n+7\}$. The required inclusion readily follows.
  \item We have $x_0<x_3$ and $x_0<x_2$. Since $d_G(x_0,x_3)=3$ and $x_3$ incomparable to $x_4$ we must have $x_0<x_4$. From $\iinc_{P'}(x_{3})\subseteq \{x_{2}, x_{4}\}$ we deduce that $x_1$ is the only element incomparable to $x_0$. From $x_1<x_3$ we deduce that $\iinc_{P'}(x_1)\subseteq \{x_0\}\cup \iinc_{P'}(x_3)$ and therefore $\iinc_P( x_{1})\subseteq \{x_0,x_2,x_4\}$. From $x_2<x_6$ and $\iinc_P( x_{6})\subseteq \{x_{5}, x_{7}\}$ we derive $\iinc_{P'}(x_2)\subseteq \{x_1,x_3,x_5,x_7\}$. Similarly, we have $\iinc_{P'}(x_4)\subseteq \{x_1,x_3,x_5,x_7\}$.
  \end{enumerate} \hfill $\Box$

From Claim \ref{claim:inequalityG} and Claim \ref{claim:pathkonig}, $\inc(P')$ is connected, locally finite and has an infinite diameter. From K\H{o}nig's Lemma,  $G'$ contains an  infinite isometric path,  hence $G$ contains an infinite induced path. This completes the proof of $(1)$.
\end{proof}

\noindent {\bf Proof of (2) of Theorem \ref{thm:infinitepath-kite}.} We break the proof into two parts.
\begin{claim} \label{claim:part1}
If $G$ is a connected incomparability graph of infinite diameter  and if the set of vertices of degree at least $3$ in $G$ has infinite diameter, then $G$  contains an infinite induced path such that the set of vertices of this path with degree at least $3$ in $G$  has an infinite diameter. \end{claim}

\noindent{\bf Proof of Claim \ref{claim:part1}.}
Let $x$ be any vertex  in $G$,   $I:=\iinc_P(x)\cup \downarrow x$ and $F:=\iinc_P(x)\cup \uparrow x$.  According to Theorem \ref{thm:orderconvex}, $I$ and $F$ are order convex and $G_{\restriction I}$ and $G_{\restriction F}$ are isometric subgraphs of $G$. Since, trivially,  $V(G)=I\cup F$,  every vertex of degree at least $3$ belongs to  $I$ or to $F$. Since the diameter in $G$ of the set of vertices of degree at least $3$ is infinite and $G_{\restriction I}$ and $G_{\restriction F}$ are isometric subgraphs we infer that the diameter in $G_{\restriction I}$ or in $G_{\restriction F}$  of the set of vertices of degree at least $3$ is infinite.  Choose  $y$ of degree at least $3$. We may assume without loss of generality that the diameter in $G_{\restriction F}$  of the set of vertices of degree at least $3$ is infinite. We start by showing that $P$ contains an infinite chain of elements whose degree is at least $3$ in $G$. Suppose constructed a sequence $y_0:=y<y_1<\dots <y_{n-1}$ of vertices of degree at least $3$ such that $d_G(y_i,y_{i+1})> 3$ for all $i\leq n-2$. Let $y_n>x_0$ be a vertex of degree at least $3$ such that $d_G(y_{n-1},y_{n})>\sum^{n-2}_{j=0}d_G(y_j,y_{j+1})$. This choice of $y_n$ is possible since the diameter in $G_{\restriction F}$ of the set of vertices of degree at least $3$ is infinite. Then $y_{n-1}$ and $y_n$ are comparable in $P$. It follows from Corollary \ref{lem:monotone-distance} that $y_{n-1}<y_n$. Hence, the sequence  $(y_i)_{i\in \NN}$ forms a chain in $P$. For all $n\in \NN$,  let $\mathrm P_n:=z_{(n,0)},z_{(n,1)},...,z_{(n,r_n)}$  be a path  in $G$ of length $r_n:=d_G(y_n,y_{n+1})$ joining $y_n$ and $y_{n+1}$. The graph  $G':=G_{\restriction \cup_{i\in \NN}V(P_i)}$ is connected and  has infinite diameter.
\begin{subclaim}\label{subclaim:part1}
$G'$ is locally finite.\end{subclaim}
\noindent{\bf Proof of Subclaim \ref{subclaim:part1}.}
It suffices to prove that for $n+2\leq m$, every vertex of $\mathrm P_n$ is  comparable to every vertex of $\mathrm P_{m}$. Let $z_{n, i}\in \mathrm P_n$ and $z_{m, j}\in \mathrm P_n$.

$\bullet$  Suppose  first $i= r_{n}-1$.
\begin{enumerate}[{(a)}]
 \item $z_{(n, r_n-1)}\leq y_{m}, z_{(m,1)}$.
 Indeed,   $z_{(n,r_n-1)}$ and $y_{m}$ are comparable, otherwise $y_{n+1}$, $z_{(n, r_n-1)}$, $y_{m}$ form a path with extremities $y_{n+1}$ and $y_{m}$ hence $d_G(y_{n+1}, y_{m})\leq 2$. This is impossible since $d_G(y_{n+1}, y_{m})\geq d_G(y_{n+1}, y_{n+2}) \geq 4$. Furthermore, $z_{(n, r_n-1)}<y_{m}$, otherwise, since $y_{n+1} <y_m$, we obtain $y_{n+1} < z_{(n, r_n-1)}$ by transitivity, while these vertices are incomparable. Similarly, $z_{(n, r_n-1)}$ and $z_{(m, 1)}$ are comparable otherwise $y_{n+1}$, $z_{(n, r_n-1)}$, $z_{(m, 1)} $, $y_{m}$ form a path with extremities $y_{n}$ and $y_{m}$ hence $d_G(y_{n+1}, y_{m})\leq 3$, while this distance is at least $4$. Necessarily, $z_{(n, r_n-1)}<z_{(m,1)}$, otherwise since $z_{(n, r_n-1)}< y_{m}$, we have $z_{m,1}< y_m$ which is impossible.

 \item By symmetry, $y_{n+1}, z_{(n, r_n-1)}\leq z_{m,1}$.

\item  $z_{(n, r_n-1)}\leq z_{(m,j)}$. We just proved it for $j=0, 1$. If $j>1$, this follows from $y_m<z_{m, i}$ by transitivity.

\end{enumerate}

$\bullet$Next, suppose  $i= r_n$.
In this case  $z_{(n,i)}= y_{n+1}$. If $j\geq 2$, we have $z_{(n,i)}= y_{n+1}<y_{m}= z_{(m, 0)} <z_{(m, j)}$. If $j=1$, this is just item (c) above.

$\bullet$ Finally, suppose that $i<r_{n}-1$. In this case,  $z_{(n, i)} <y_{n+1}<z_{(m,j)}$. \hfill $\Box$

Since $G'$ is connected, locally finite and has an infinite diameter,  K\H{o}nig's Lemma ensures that it contains an  infinite isometric path $\mathrm P_\infty$. We claim that $\mathrm P_\infty$ contains an infinite number of vertices of degree at least $3$ in $G$. Clearly, $V(\mathrm P_\infty)$ meets infinitely many $\mathrm P_i$'s. For each $i\in \NN$ let $j_i\in V(\mathrm P_i)$ be the largest such that $z_{(i,j_i)}\in V(\mathrm P_\infty)$. Then the degree of $z_{(i,j_i)}$ is at least $3$ in $G$. Indeed, if $z_{(i,j_i)}\in \{y_i,y_{i+1}\}$, then we are done. Otherwise $z_{(i,j_i)}$ is not an end vertex of $\mathrm P_i$. Then $z_{(i,j_i)}$ must have a neighbour in $\mathrm P_{\infty}$ which is not in $\mathrm P_i$ and therefore must have degree three. So far we have proved that $G'$ contains an  infinite isometric path $\mathrm P_\infty$ containing infinitely many vertices of degree at least $3$. Hence, $G$ contains an  infinite induced path $\mathrm P_\infty$ containing infinitely many vertices of degree at least $3$. This proves our claim. \hfill $\Box$

\begin{claim}\label{claim:part2}If  $G$ is a connected incomparability graph  containing an infinite induced path such that the set of vertices of this path with degree at least $3$ in $G$  has an infinite diameter    then $G$ contains either a caterpillar or a kite.
\end{claim}

\noindent{\bf Proof of Claim \ref{claim:part2}.}
Let $(x_n)_{n\in \NN}$ be a sequence of vertices of $G$  with $(x_n, x_{n+1})\in E(G)$ for $n\in \NN$ forming an infinite induced path $\mathrm P_{\infty}$.  Suppose that this path contains infinitely many vertices with degree at least $3$ in $G$ forming a set of infinite diameter in $G$.
\begin{subclaim}\label{subclaim:part2}There is an infinite sequence $(y_n)_{n}$   of  vertices in $V\setminus \mathrm P_{\infty}$ forming an independent set and a family of disjoint intervals $I_n:= [l(n), r(n)]$ of $\NN$ such that  $\{l(n), r(n)\}  \subseteq B_G(y_n, 1)\cap \mathrm P_{\infty} \subseteq I_n$ for all $n\in \NN$.
\end{subclaim}
\noindent{\bf Proof of Subclaim \ref{subclaim:part2}.} Pick $x_{i_0}\in \mathrm P_{\infty}$ with degree at least $3$ in $G$ and set $y_0$ arbitrary in $B_G(x_{i_0}, 1)\setminus \mathrm P_{\infty}$. According to Lemma \ref{ball-infinitepath2} the ball $ B_G(y_0, 1)$ contains only finitely many vertices of $\mathrm P_{\infty}$. Let $l(0)$, resp., $r(0)$  be the  least, resp.,  the largest  integer $k$ such that  $x_k\in B_G(y_0, 1)$.
Let $n >0$. Suppose $(y_m)_m$, $I_m:= [l(m), r(m)]$ be defined for $m<n$.
 By Lemma \ref{ball-infinitepath2}, $\mathrm P_{\infty} \cap (\bigcup_{m<n}B_G(y_m, 2))$ is finite, hence there is a vertex  $x_{i_n}\in \mathrm P_{\infty}$  with degree at least $3$ such that every vertex in the  infinite subpath of $\mathrm P_{\infty}$ starting at $x_{i_n}$ is at distance at least $3$ of any $y_m$. Pick $y_n\in B(x_{i_n}, 1)\setminus \mathrm P_{\infty}$ and set $I_n=[l(n), r(n)]$ where $l(n)$, resp., $r(n)$  be the  least, resp.,  the largest  integer $k$ such that  $x_k\in B_G(y_n, 1)$. \hfill $\Box$

 In order to complete the proof of Claim \ref{claim:part2} we show that the graph $G'$ induced on $\mathrm P_{\infty} \bigcup \{y_n: n\in \NN\}$ contains a caterpillar or a kite. For that, we classify the vertices $y_n$.
We say that $y_n$ has \emph{type} $(0)$ if $l(n)=r(n)$ (that is $y_n$ has just one neighbour on $\mathrm P_{\infty}$.
If the set $Y_0$ of  vertices of type $(0)$ is infinite then trivially $G_{\restriction \mathrm P_{\infty} \bigcup Y_0}$ is a   caterpillar (see Figure \ref{fig:comb-kite}). We say that $y_n$ has \emph{type} $(1)$ if  $r(n)= l(n)+1$. Again, trivially, if the set $Y_1$ of  vertices of type $(1)$ is infinite then  $G_{\restriction \mathrm P_{\infty} \bigcup Y_1}$ is a   kite of \emph{type} $(1)$. We say that $y_n$ has \emph{type} $(2)$ if $r(n)=l(n)+2$. It has \emph{type} $(2.1)$ if  $(y(n), x_{l(n)+1})\in E(G)$ while it has \emph{type} $(2.2)$ if $(y(n), x_{l(n)+1})\not \in  E(G)$.  If for $i=1,2$ the set $Y_{2.i}$ of  vertices of type $(2.i)$ is infinite then  $G_{\restriction \mathrm P_{\infty} \bigcup Y_{2.i}}$ is a   kite of type $(i+1)$ (see Figure \ref{fig:comb-kite}). We say that $y_n$ has type $(3)$ if $r(n)\geq l(n)+3$. It has type $(3.1)$ if $(y(n), x_{l(n)+1})\in E(G)$ while it has type $(3.2)$ if  $(y(n), x_{l(n)+1})\not \in E(G)$. If the set  $Y_{3.i}$ of  vertices of type $3.i$ is infinite delete from $\mathrm P_{\infty}$ the set $Y:= \bigcup_{n\in Y_{3.i}}\{x_{m}: m\in \{l(n+2,\dots,  r(n)-1 \}$. Then   $G_{\restriction (\mathrm P_{\infty} \bigcup Y_{3.i})\setminus Y}$ is  a   kite of type $(2)$ if $i=1$ or a caterpillar if $i=2$ (see Figure \ref{fig:comb-kite}).\hfill $\Box$

\section{Example  \ref{thm:noisometric}} \label{section:proof-thm:noisometric}

We define the poset satisfying the conditions stated in  Example \ref{thm:noisometric}.

For a poset $P=(V,\leq)$ we set for every $x\in V$ we set $\iinc_P(x):= \{y\in V: x\parallel y\}$.\\

Let $P:=(X,\leq)$ be the poset defined on $X:=\NN\times \NN \times \{0,1\}$ as follows. We let $(m,n,i)\leq (m',n',i')$ if
\[i=i' \mbox{ and } [ n<n' \mbox{ or } (n=n' \mbox{ and } m\leq m') ],\]
\[\mbox{or}\]
\[i\neq i' \mbox{ and } [n+1<n' \mbox{ or } (n+1=n' \mbox{ and } m\leq m')].\]
We set $A_n:=\{(m,n,1) : m\in \NN\}$ for all $n\geq 0$ and $B_n:=\{(m,n,0) : m\in \NN\}$ and note that $\cup_{n\in \NN}A_n$ and $\cup_{n\in \NN} B_n$ are two total orders of order type $\omega^2$. In particular $P$ is coverable by two chains and hence has width two. \\
\textbf{Claim 1:} $\leq$ is an order relation.\\
Reflexivity and antisymmetry are obvious. We now prove that $\leq$ is transitive. Let $(m,n,i)$, $(m',n',i')$, $(m'',n'',i'')$ be such that $(m,n,i)\leq (m',n',i')\leq (m'',n'',i'')$.  Note that since $\{i,i',i''\}\subseteq \{0,1\}$ at least two elements of $\{i,i',i''\}$ are equal. if $i=i'=i''$ then clearly $(m,n,i)\leq (m'',n'',i'')$. Next we suppose that there are exactly two elements of $\{i,i',i''\}$ that are equal. There are three cases to consider.\\
$\bullet$ $i=i'$.\\
Since $(m,n,i)\leq (m',n',i')$ we have
\begin{equation}\label{1 i=i'}
  n<n' \mbox{ or } (n=n' \mbox{ and } m\leq m').
\end{equation}
Since $i'\neq i''$ and $(m',n',i')\leq (m'',n'',i'')$ we have
\begin{equation}\label{2 i' not i''}
n'+1<n'' \mbox{ or } (n'+1=n'' \mbox{ and } m'\leq m'').
\end{equation}

If $n+1<n''$, then since $i\neq i''$ it follows that $(m,n,i)\leq (m'',n'',i'')$. Suppose $n''\leq n+1$. If $n'+1<n''$, then $n'<n$. It follows from (\ref{1 i=i'}) that $n=n'$ and hence $n+1<n''$ proving that $(m,n,i)\leq (m'',n'',i'')$. Else, $n''\leq n'+1$. It follows from (\ref{2 i' not i''}) that $n'+1=n''$ and $m'\leq m''$. If $n<n'$, then $n+1<n''$ and once again we have $(m,n,i)\leq (m'',n'',i'')$. Otherwise it follows from (\ref{1 i=i'}) that $n=n'$  and $m\leq m'$. Hence, $n+1=n''$ and $m\leq m''$ proving that $(m,n,i)\leq (m'',n'',i'')$. \\

$\bullet$ $i=i''$.\\
Since $(m,n,i)\leq (m',n',i')$ and $i\neq i'$ we have
\begin{equation}\label{3 i not i'}
n+1<n' \mbox{ or } (n+1=n' \mbox{ and } m\leq m').
\end{equation}

Since $(m',n',i')\leq (m'',n'',i'')$ and $i'\neq i''$ we have
\begin{equation}\label{4 i' not i''}
n'+1<n'' \mbox{ or } (n'+1=n'' \mbox{ and } m'\leq m'').
\end{equation}
We prove that $n<n''$. We suppose $n''\leq n$ and we argue to a contradiction. We Claim that none of $n+1<n'$ and $n'+1<n''$ can hold. Indeed, suppose $n+1<n'$. Then $n''<n'$ and hence $n'+1<n''$ cannot be true. It follows from (\ref{4 i' not i''}) that $n'+1=n''$. But then $n''=n'+1>n'>n''$ which is impossible. Now suppose $n'+1<n''$. Then $n'+1<n<n+1<n'$ and this is impossible. It follows from (\ref{3 i not i'}) and (\ref{4 i' not i''}) that  $n+1=n' \mbox{ and } m\leq m'$ and $n'+1=n'' \mbox{ and } m'\leq m''$. Hence, we have proved our claim that none of $n+1<n'$ and $n'+1<n''$ can hold. It follows from (\ref{3 i not i'}) and (\ref{4 i' not i''}) that $n+1=n'$ and $n'+1=n''$, and in particular $n+2=n''$. This contradicts $n''\leq n$. Hence, $n<n''$ and therefore $(m,n,i)\leq (m'',n'',i'')$ since $i=i''$.\\
$\bullet$ $i'=i''$.\\
Since $(m,n,i)\leq (m',n',i')$ and $i\neq i'$ we have
\begin{equation}\label{5 i not i'}
n+1<n' \mbox{ or } (n+1=n' \mbox{ and } m\leq m').
\end{equation}

\noindent Since $(m',n',i')\leq (m'',n'',i'')$ and $i'= i''$ we have
\begin{equation}\label{6 i'=i''}
  n'<n'' \mbox{ or } (n'=n'' \mbox{ and } m'\leq m'').
\end{equation}
If $n+1<n''$, then $(m,n,i)\leq (m'',n'',i'')$ since $i\neq i''$. We Claim that none of $n+1<n'$ and $n'<n''$ can hold. Suppose $n+1<n'$. Then $n''<n$ and it follows from (\ref{6 i'=i''}) that $n'=n''$. But then $n''\leq n+1<n'=n''$ which is impossible. Suppose $n'<n''$. Then $n'<n+1$ and it follows from (\ref{5 i not i'}) that $n+1=n'$. But then $n+1=n'<n''<n+1$ which is impossible. Hence, none of $n+1<n'$ and $n'<n''$ can hold. It follows from (\ref{5 i not i'})  and (\ref{6 i'=i''}) that $(n+1=n' \mbox{ and } m\leq m')$ and $(n'=n'' \mbox{ and } m'\leq m'')$. Therefore, $(n+1=n'' \mbox{ and } m\leq m'')$ proving that  $(m,n,i)\leq (m'',n'',i'')$ as required.\\
\textbf{Claim 2:} Let $j\in \NN$. Then for all $x\in A_j$, $|B_{\inc(P)}(x,1)\cap B_{j+1}|$ is finite.\\
Let $x:=(m,j,1)\in A_j$. Then $B_{\inc(P)}(x,1)\cap B_{j+1}=\{(k,j+1,0): 0\leq k\leq m-1\}$.\\
\textbf{Claim 3:} Let $j\in \NN$. Then for all $x\in B_j$, $|B_{\inc(P)}(x,1)\cap A_{j+1}|$ is finite.\\
Let $x:=(m,j,0)\in B_j$. Then $B_{\inc(P)}(x,1)\cap A_{j+1}=\{(k,j+1,1): 0\leq k\leq m-1\}$.\\
\textbf{Claim 4:} Let $j\in \NN$. Then for all $x\in A_j$ and for all $y\in B_{\inc(P)}(x,1)\cap B_{j+1}$, $|B_{\inc(P)}(y,1)\cap A_{j+2}|<|B_{\inc(P)}(x,1)\cap B_{j+1}|$.\\
Let $x:=(m,j,1)\in A_j$. It follows from Claim 2 that $|N(x)\cap B_{j+1}|=m$. Let $y\in B_{\inc(P)}(x,1) \cap B_{j+1}$, say $y=(m',j+1,0)$ and note that $m'<m$. Then it follows from Claim 3 that $|B_{\inc(P)}(x,1)\cap B_{j+2}|=m'$. Since $m'<m$ we are done.\\
\textbf{Claim 5:} Let $j\in \NN$. Then for all $x\in B_j$ and for all $y\in B_{\inc(P)}(x,1)\cap A_{j+1}$, $|B_{\inc(P)}(y,1)\cap B_{2+1}|<|B_{\inc(P)}(x,1)\cap A_{j+1}|$.\\
Symmetry and Claim 4.\\
\textbf{Claim 6:} If there exists and infinite isometric path $(x_i)_{i\in \NN}$ in $\inc(P)$ starting at $x_0=(0,0,1)$, then $x_{2n}\in A_{2n-1}$ and $x_{2n+1}\in  B_{2n}$.\\
From $B_0=\iinc_P(x_0):= \{y\in X: y\; \text{incomparable to}\; x \; \text{in}\; P\}$ follows that $x_1\in B_0$. Suppose for a contradiction that $x_2\not \in A_1$. Then $x_2\in A_0$. In this case $x_3\in B_1$ (this is because $x_0<x_3$). But then $x_4\in A_2$ because otherwise $x_4\in A_1$ and hence the distance from $x_0$ to $x_4$ would be two which is not possible. By the same token $x_5\in B_3$ and more generally $x_{2n+1}\in B_{2n-1}$ and $x_{2n-2}\in A_{n}$. This is impossible. Indeed, suppose $x_2=(i,0,0)$ then $x_3=(j,1,1)$ with $j<i$ and then $x_4=(k,2,0)$ with $k<j$. Continuing this way we have a decreasing sequence of nonnegative integers.\\
\textbf{Claim 7:} Let $y\in B_{\inc(P)}(x_0,1)\cap B_0$. Then the lengths of isometric paths starting at $x_0$ and going through $y$ is bounded.\\
Follows from Claims 4, 5 and 6.\\
We conclude that there is no isometric path in $\inc(P)$ starting at $(0,0,1)$. It follows from Theorem \ref{thm:polat} that $\inc(P)$ has no isometric infinite path.

\section{Interval orders: A proof of Theorem \ref{thm:intervalorder-isometric} and Example \ref{thm:intervalorder-non-isometric}}\label{section:intervalorders}

We recall that an order $P$ is an {\it interval order} if
$P$ is isomorphic to a subset $\mathcal J$ of the set $Int(C)$ of
non-empty intervals of a chain $C$, ordered as follows: if $I, J\in
Int(C)$, then
\begin{equation}\label{ordre-sur-intervalles-recall}
I<J \mbox{  if  } x<y  \mbox{  for every  } x\in I  \mbox{  and
every  } y\in J.
\end{equation}

The following proposition encompasses  some known equivalent properties of interval orders. Its proof is easy and is left to the reader.

\begin{proposition}Let $P:=(V,\leq)$ be a poset. The following propositions are equivalent.
\begin{enumerate}[(i)]
  \item $P$ is an interval order.
  \item $P$ does not embed $2\oplus 2$.
  \item The set $\{(\downarrow{x})\setminus \{x\} : x\in V\}$ is totally ordered by set inclusion.
  \item The set $\{(\uparrow{x})\setminus \{x\} : x\in V\}$ is totally ordered by set inclusion.
\end{enumerate}
\end{proposition}

\begin{lemma}\label{lemma:neigbour-antichain}Let $P=(V,\leq)$ be an interval order and $x\in V$. Then the neighbours of $x$ (in $\inc( P)$) that lay on an induced path of length at least two in $\inc( P)$  and starting at $x$ and whose vertices are in $\iinc_P( x)\cup \uparrow x$ form an antichain in $P$.
\end{lemma}
\begin{proof} Let $x:=x_0, x_1,\dots, x_{n}$  and  $x:=x'_0, x'_1,\dots, x'_{n'}$ be two induced paths in $\inc( P)$ with $n,n'\geq 2$ and whose vertices are in $\iinc_P( x)\cup \uparrow x$. Note that necessarily $x<x_2$ and $x<x'_2$. Suppose for a contradiction that $x_1$ and $x'_1$ are comparable. Suppose $x_1<x'_1$. Since $x<x_2$ and $x_1$ is incomparable to $x$ and to $x_2$ and $x$ is incomparable to $x'_1$ and $P$ is an interval order we infer that $x'_1$ is comparable to $x_2$ and hence $x<x'_1$ or $x_1<x_2$, which is impossible . The case $x'_1<x_1$ can be dealt with similarly by considering the comparabilities $x'_1<x_1$ and $x<x'_2$.
\end{proof}

\begin{proof}(Of Theorem \ref{thm:intervalorder-isometric})
Let $x_0\in P$ and set $I_0:=\iinc_P( x_0)\cup \downarrow x_0$ and $F_0:=\iinc_P( x_0)\cup \uparrow x_0$. Clearly, $V(G)=I_0\cup F_0$. Furthermore, since the diameter of $G$ is infinite and $G_{\restriction I_0}$ and $G_{\restriction F_0}$ are connected graphs we infer that the diameter in $G_{\restriction I_0}$ or in $G_{\restriction F_0}$  is infinite. We may assume without loss of generality that the diameter of $G_0:=G_{\restriction F_0}$ is infinite. Hence, the lengths of isometric paths in $G_0$ starting at $x_0$ are unbounded.\\
\textbf{Claim 1:} There exists $x_1\in \iinc_P( x_0)$ such that the lengths of isometric paths in $G_0$ starting at $x_0$ and going through $x_1$ are unbounded.\\
Since the antichains of $P$ are finite, there are only finitely many neighbours of $x_0$ in $G_0$ laying on isometric paths starting at $x_0$ and of length at least two. Hence there must be a neighbour $x_1$ of $x$ in $G_0$ such that the lengths of isometric paths in $G_{\restriction {F_0}}$ starting at $x_0$ and going through $x_1$ are unbounded.\\

Now suppose constructed an isometric path $x_0,...,x_n$ such that  $x_i< x_j$ for all $j-i\geq 2$ and that the lengths of isometric paths starting at $x_0$ and going through $x_0,...,x_n$ are unbounded. From Lemma \ref{lemma:neigbour-antichain} we deduce that there are only finitely many neighbours of $x_n$ that lay on such isometric paths. Applying Claim 1 to $x_n$ we deduce that there exists $x_{n+1}>x_{n-1}$ such that $x_0,...,x_n,x_{n+1}$ is an isometric path of length $n+1$.
\end{proof}

We now proceed to the proof of Example \ref{thm:intervalorder-non-isometric}.

\begin{proof}
We totally order the set $\NN\times \NN$ as follows: $(n,m)\leq (n', m')$ if $m<m'$ or ($m=m'$ and $n\leq n'$). Consider the set $Q$ of intervals $X_{n, m}:=[(n, m), (n, m+1)[$ ordered as in (\ref{ordre-sur-intervalles}) above and set $G:=\inc(Q)$. Then $X_{n, m} \leq X_{n', m'}$ if and only if $m+1<m'$ or ($m+1=m'$ and $n\leq n'$). Equivalently, $\{X_{n, m},X_{n', m'}\}$ is an edge of $G$ if and only if $m=m'$ or ($m'=m+1$ and $n'<n$) or ($m=m'+1$ and $n<n'$). \\
\textbf{Claim 1:} $G$ is connected and has infinite diameter.\\
Let $X_{n, m}$ and $X_{n', m'}$ be two elements of $Q$ so that $n\leq n'$. We may suppose without loss of generality that $X_{n, m}\cap X_{n', m'}=\varnothing$. We may suppose without loss of generality that $m< m'$.  Consider the sequence of intervals $X_{n,m},X_{n',m},X_{n+1,m+1}, X_{n,m+2},X_{n',m+2},...,X_{n',m'}$. This is easily seen to be a path in $G$ proving that $G$ is connected.\\ 
\textbf{Claim 2:} $G$ has no isometric infinite path starting at $X_{0,0}$.\\
Let $X_{0,0}=:Y_0, \dots Y_r\dots$ be an isometric path. Then $Y_1=X_{n_1,0}$ for some $n_1\in \NN$. Now $Y_2$ must intersect $Y_1$ but not $Y_0$. Hence, $Y_2=X_{n_2,1}$ for some $n_2<n_1$. Now $Y_3$ must intersect $Y_2$ but not $Y_1$. Suppose $Y_3=X_{n',1}$. Then $n_1<n'$. But then $X_{n'+1,0}$ intersects $Y_3$ and $Y_0$ and therefore the distance in $G$ between $Y_0$ and $Y_3$ is two contradicting our assumption that $X_{0,0}=:Y_0, \dots Y_n\dots$ is isometric. Hence, we must have $Y_3=X_{n_3,2}$ for some $n_3<n_2$. An induction argument shows that $Y_r=X_{n_r,r-1}$ with $n_r<n_{r-1}<...<n_1$. Since there are no infinite strictly decreasing sequences of positive integers the isometric path $X_{0,0}=:Y_0, \dots Y_r\dots$ must be finite. This completes the proof of Claim 2.\\
It follows from Theorem \ref{thm:polat} that $G$ has no isometric path.
\end{proof}

\section*{Acknowledgement} We are indebted to an anonymous referee for their careful examination of the paper and for their comments which improved its presentation.

\end{document}